\documentclass[11pt,reqno]{amsart}

\usepackage[T1]{fontenc}
\usepackage[utf8]{inputenc}
\usepackage{lmodern}
\usepackage[margin=1in]{geometry}
\usepackage{microtype}
\usepackage{amsmath,amssymb,amsthm,mathtools,mathrsfs}
\usepackage{bm}
\usepackage{enumitem}
\usepackage{hyperref}
\usepackage[nameinlink,noabbrev]{cleveref}

\usepackage{biblatex} 
\newtheorem{theorem}{Theorem}[section]
\newtheorem{lemma}[theorem]{Lemma}
\newtheorem{corollary}[theorem]{Corollary}
\theoremstyle{remark}
\newtheorem{remark}[theorem]{Remark}

\numberwithin{equation}{section}
\title[Stochastic heat equation with Hölder diffusion coefficient]{Stochastic heat equation with nondegenerate Hölder diffusion coefficient: uniqueness below the three-fourth threshold}
\author{Yi Han}
\address{Department of Mathematics, Massachusetts Institute of Technology, Cambridge, MA
}
\email{hanyi16@mit.edu}
\date{}

\begin{document}

\begin{abstract}
    We consider stochastic heat equation (SHE) defined on 1-d torus $\mathbb{T}$ of the form 
    $$
\partial_t u=\Delta u+g(u)\dot{W},
    $$where $\dot{W}$ is a space-time white noise and $g$ is a real-valued function which is uniformly elliptic (i.e., $|g|$ is uniformly bounded away from 0), and is globally $\beta$-Hölder continuous for some $\beta\in(0,1)$. We prove that weak uniqueness holds as long as $\beta>\frac{2}{3}$. The same uniqueness holds for vector-valued solutions where the coefficient $G$ has the same dimension as the white noise. Previously, uniqueness of solutions to the SHE with Hölder diffusion coefficient was only established for $\beta>\frac{3}{4}$ via a Yamada–Watanabe argument by Mytnik and Perkins \cite{mytnik2011pathwise} without assuming $g$ is nonzero. And when $\beta<\frac{3}{4}$, Mueller, Mytnik and Perkins \cite{mueller2014nonuniqueness} constructed a  non-unique SPDE example satisfying $g(0)=0$. A later generalized coupling argument for nondegenerate $g$ also stopped at the same threshold $\frac{3}{4}$. Our result shows that uniform ellipticity of $g$ restores uniqueness to SHEs in the Hölder regime where the same SHE with non-elliptic $g$ and the same Hölder regularity are often non-unique in law. This constitutes the first general class of SHE weak uniqueness results in the $\beta\in(\frac{2}{3},\frac{3}{4}]$ regime.
\end{abstract}
\maketitle

\section{Introduction}

Let \(\mathbb T\) be the unit torus in one dimension. Consider the following stochastic heat equation (SHE)

\begin{equation}
\partial_tu=\Delta u+g(u)\,\dot W,
\qquad u(0)=u_0\in L^2(\mathbb{T}),
\label{1.1}
\end{equation}
where $g:\mathbb{R}\to\mathbb{R}$ is a real-valued function. For any $\beta\in(0,1)$ we define $[g]_{C^\beta}$ via
$$
[g]_{C^\beta}=\sup_{x\neq y\in\mathbb{R}}\frac{|g(x)-g(y)|}{|x-y|^\beta}
.$$

The main result of the paper is the following theorem:
\begin{theorem}
\label{theorem1.11}
Assume that the diffusion coefficient $g$ satisfies, for a constant  $c_g>0$:
\begin{equation}\label{1.2}
g\in C^\beta(\mathbb R),
\qquad
\frac23<\beta<1,
\qquad
0<c_g\le |g(z)|\quad \forall z\in\mathbb{R},
\end{equation}
where $C^\beta(\mathbb{R})$ is the space of real-valued continuous functions $g$ that satisfy $[g]_{C^\beta}<\infty$. Then for every deterministic \(u_0\in L^2(\mathbb T)\), weak uniqueness
holds for \eqref{1.1}.  More precisely, on every finite time interval $[0,T]$, any two
weak mild solutions to \eqref{1.1} have the same law on
\(C([0,T];L^2(\mathbb T))\). 
\end{theorem}

In this paper, all uniqueness statements are proven for uniqueness in law rather than pathwise uniqueness of SHEs. The current proof is based on law comparison and does not have a natural path to pathwise uniqueness. We do not assume $g$ is bounded from above, but the Hölder assumption $[g]_{C^\beta}<\infty$ on $g$ means that $g$ has at most linear growth. For example, we can take $g(x)=1+|x|^\beta,\quad\beta>\frac{2}{3}$ as the diffusion coefficient for the SPDE \eqref{1.1}.

\subsection{Uniqueness of martingale solutions and SPDE with Hölder continuous noise coefficient}
We now put Theorem \ref{theorem1.11} into the general context of SPDE weak solutions. First we note that the existence of a probabilistic weak solution to \eqref{1.1} follows from a very standard tightness argument (see \cite{gkatarek1994weak}) whenever the coefficients of the SHE are continuous. However, proving uniqueness in law for \eqref{1.1} for non-locally-Lipschitz coefficients $g$ has been a very difficult task in the SPDE context.

Taking a step back, the same uniqueness question can be asked for finite dimensional diffusion processes of the form $dX_t=\sigma(X_t)dW_t$ for any continuous function $\sigma:\mathbb{R}^d\to\mathbb{R}^{d\times d}$ and a $d$-dimensional Brownian motion $W_t$. When $\sigma$ is uniformly elliptic and in the one dimensional setting, the solution $X_t$ is already known to be unique in law (see Chapter 6 of \cite{stroock2007multidimensional}) by a time change argument for Brownian motion, and no Hölder continuity of the coefficient $\sigma$ is required. These uniqueness statements for elliptic $\sigma$ can also be proven via the parametrix method (see \cite{AST_2009__327__47_0}) when $\sigma$ is Hölder continuous and multi-dimensional, or via elliptic PDE methods including the finite-dimensional Schauder theory. However, none of these methods admit a complete counterpart for infinite dimensional SPDEs. When $\sigma$ is allowed to vanish somewhere, the uniqueness picture changes completely. Take $d=1$, then the Yamada–Watanabe pathwise uniqueness criterion \cite{yamada1971uniqueness} shows uniqueness for $\beta\ge \frac{1}{2}$-Hölder continuous $\sigma$ and shows nonuniqueness for $\beta<\frac{1}{2}$.

For infinite dimensional SHE driven by white noise and merely Hölder continuous diffusion coefficient $g$, some of the major developments to date can be summarized as follows:

\begin{enumerate}
    \item Zambotti \cite{zambotti2000analytic} proved uniqueness in law for a class of martingale problems where the diffusion coefficient $g$ is a trace-class Hölder perturbation of a fixed Ornstein-Uhlenbeck covariance. The proof uses Schauder estimates but does not cover the pointwise coefficient SHE \eqref{1.1}.
    \item Mytnik and Perkins \cite{mytnik2011pathwise} proved pathwise uniqueness for 1-d SHE (on $\mathbb{R}$) for $\beta$-Hölder continuous coefficients $g$ whenever $\beta>\frac{3}{4}$ via a Yamada-Watanabe approach. No ellipticity assumption on $g$ was imposed.
    \item Mueller, Mytnik and Perkins \cite{mueller2014nonuniqueness} showed that, in the same context as \cite{mytnik2011pathwise} but when $\beta<\frac{3}{4}$, uniqueness in law no longer holds for general $g$. Their counterexample takes $g(x)=|x|^\gamma$, so that $g$ is not elliptic. Thus for uniformly elliptic coefficient $g$, the uniqueness of \eqref{1.1} is not disproved for $\beta<\frac{3}{4}$.
    \item Bass and Perkins \cite{bass2012uniqueness} proved uniqueness for a class of martingale problems where $g$ is elliptic and $\frac{1}{2}+\epsilon$-Hölder continuous in an enhanced Hölder norm. Their assumption was that the Hölder constant of $g$ should decay much faster when tested against Fourier modes of higher frequencies, but the standard coefficient model \eqref{1.1} does not satisfy this fast decay property of $g$ at higher modes.
    \item The author \cite{han2025exponential} adapted the generalized coupling method of \cite{kulik2020well} and proved weak uniqueness for a wide class of SPDEs with elliptic diffusion coefficient. When specialized to the 1-d SHE, the paper also gives the $\beta>\frac{3}{4}$ threshold (but requires ellipticity of $g$). This method was generalized to the setting of stochastic wave equations in \cite{han2024stochastic}.
\end{enumerate}

As a historical note, part of the study of SHEs with Hölder $g$ is motivated by the well-known example of a
Dawson--Watanabe super-Brownian motion (see for example \cite{Dawson1993MeasureValued}, \cite{Perkins2002DawsonWatanabe}, \cite{BassPerkins2001SuperBrownianMP}). In its SPDE formulation, this corresponds to a SHE with noise coefficient $\sqrt{u}$. The proof of uniqueness for this model is quite special, as the law is determined by the branching property of an interacting particle system via the associated
log-Laplace functional. Therefore, previous research can be summarized into two main streams: the first stream considers $g$ that can possibly take the value zero, and is partly motivated by this super-Brownian motion example. The second stream considers $g$ that is uniformly elliptic, and is partly motivated by the finite-dimensional uniqueness result for martingale problems with elliptic $g$. This paper lies within the second stream.

As illustrated in the above list, within the second stream of elliptic $g$, previously there was no general result that can go beyond the Hölder regime $\beta>\frac{3}{4}$ which is the critical threshold suggested by known results for non-elliptic $g$. Theorem \ref{theorem1.11} finally proves weak uniqueness in the regime $\frac{3}{4}\geq\beta>\frac{2}{3}$
where ellipticity of $g$ really matters for uniqueness. This is the first general solution theory for SHE in the Hölder regime where ellipticity of $g$ becomes decisive for uniqueness.

\subsection{Generalizations: vector field solutions, drifts, and more}

We then present three generalizations of Theorem \ref{theorem1.11}. First, the uniqueness proof extends to vector-valued solutions without much difficulty:

\begin{theorem}[Finite-dimensional vector-field SHEs]
\label{thm:vector-valued}
Let \(d\geq1\), \(\beta\in(2/3,1)\), and let
\[
    G:\mathbb R^d\to\mathbb R^{d\times d}
\]
be a measurable mapping that satisfies the following two conditions:
\begin{enumerate}
    \item For some constant \(c_G>0\),
\[
    G(z)G(z)^{*}\geq c_G^2 I_d,
    \qquad z\in\mathbb R^d,
\]
\item    For some constant \([G]_{C^\beta}<\infty\),
\begin{equation}
    \|G(z)-G(z')\|_{\mathrm{HS}}
       \leq [G]_{C^\beta}\|z-z'\|^\beta,
    \qquad \forall z,z'\in\mathbb R^d.
    \label{HOLDER}
\end{equation}
\end{enumerate}

Consider the following vector-valued stochastic heat equation
$$
    \partial_t U
      =\Delta U+G(U)\dot W,
      \qquad
      U(0)=u_0\in L^2(\mathbb T;\mathbb R^d),
$$
where \(W=(W^1,\ldots,W^d)\) is a \(d\)-dimensional
space--time white noise.

Then weak uniqueness holds for the above equation in
$
    C\bigl([0,T];L^2(\mathbb T;\mathbb R^d)\bigr)
$
for every \(T>0\).
\end{theorem}

We can also add a globally Lipschitz drift to the SHE:
\begin{corollary}[Globally Lipschitz drift]
\label{cor:lipschitz-drift}
Let $b:\mathbb{R}^d\to\mathbb{R}^d$ be a measurable function. Let $G$ satisfy the assumption in Theorem~\ref{thm:vector-valued}. Assume that $b$ is globally Lipschitz continuous:
\[
    |b(z)-b(z')|
       \leq L_b|z-z'|,
    \qquad\forall z,z'\in\mathbb{R}^d.
\] Then the following SHE with drift
\begin{equation}\label{sheu}\partial_tU
      =\Delta U+b(U)+G(U)\dot W,\quad U(0)=u_0\in L^2(\mathbb{T},\mathbb{R}^d)
\end{equation}
has a unique weak mild solution.
In particular, \(b\) is allowed to have linear growth.
\end{corollary}
We can also add a bounded measurable drift via Girsanov, since the diffusion is elliptic:

\begin{remark}[Bounded measurable drift]
\label{cor:bounded-normalized-drift} Via a standard Girsanov uniqueness argument, the following result should hold with a direct proof:

Suppose that the assumptions of Theorem~\ref{thm:vector-valued} hold and for a measurable map $b:\mathbb{R}^d\to\mathbb{R}^d$ define
\[
    q(z)
      :=G(z)^*\bigl(G(z)G(z)^*\bigr)^{-1}b(z).
\]
Then if \(q(z)\) is a bounded measurable function on $\mathbb{R}^d$, then the SHE \eqref{sheu} with drift
\(b\) has a unique weak mild solution on every finite time
interval.
\end{remark}
Since $G$ is elliptic, Remark~\ref{cor:bounded-normalized-drift} already covers
any bounded drift coefficient $b$. More general linearly growing drifts $b$ with no continuity can possibly be
treated by a localized Girsanov--Bene\v{s} argument (see for example
\cite[Theorem~2.1 and Lemma~6.1]{KlebanerLiptser2014}), but this proof requires
uniform energy estimates under the stopped changes of measure and is thus
not presented here.

\begin{remark}
    Although Theorem \ref{theorem1.11} and \ref{thm:vector-valued} and the two corollaries are stated on the torus, the same proof is also expected to work when the state space is the interval $[0,1]$ and we endow $L^2([0,1])$ with either Dirichlet or Neumann boundary condition (after replacing torus heat kernel identities by relevant bounds). Also, we have only used time-independent coefficients $g$ and $b$ for the SHE, but a similar uniqueness argument is expected to extend to time and space-dependent coefficients $g(t,x,\cdot),b(t,x,\cdot)$ whenever the quantitative estimates on the coefficients are uniform in $t>0,x\in\mathbb{T}$.
\end{remark}

Beyond these immediate extensions, we believe that the following extensions can also be made as a further development of the solution theory in this paper. However, all entries in the following lists are future outlooks instead of verified theorem statements:

\begin{enumerate}
\item One can recast the solution theory in this paper into a more abstract formulation of SPDEs on a separable Hilbert space in the sense of \cite{han2025exponential}, \cite{han2024stochastic}. We may also replace the Laplacian with a fractional Laplacian, and derive the corresponding threshold for the Hölder index $\beta$ depending on the power of the Laplacian.

    \item Instead of $L^2(\mathbb{T})$ initial condition, one can start the SHE at a Dirac delta initial condition $\delta_{x},x\in\mathbb{T}$. The uniqueness proof should also be generalizable to general distribution-valued initial laws whenever the solution immediately becomes $L^2$ after the heat semigroup action. 
    \item It might be possible to adapt the proof of Theorem \ref{theorem1.11} to stochastic wave equations, in the sense of \cite{han2024stochastic}. The coupling method probably transfers to second-order wave semigroups.
    \item It might be possible to derive an ergodic theory for the SHE \eqref{1.1} in the sense of \cite{han2025exponential}. However, the difference is that \cite{han2025exponential} uses a Wasserstein distance defined from the metric on $L^2(\mathbb{T})$, but we may need to use a Wasserstein distance defined from the much weaker metric on the Sobolev space $H^{-s}$ for some $s>1/2$.
    \item As a more ambitious goal, one can expect to extend Theorem \ref{theorem1.11} to SHE on the whole real line $\mathbb{R}$. Significant functional analytic work might be needed for this direction. 
\end{enumerate}

\subsection{Main ideas of the proof}

When one first tries to prove weak uniqueness of \eqref{1.1}, a natural idea may be to expand the solution into a convergent series, which is reminiscent of the parametrix method (see \cite{AST_2009__327__47_0}). However, in infinite dimensions the canonical density does not exist and convergence of the series expansion is extremely hard to guarantee from this Hölder $g$. 

A better method may still be to compare the law of the candidate solution $U$ to a SHE solution $V^n$ with smooth coefficients $g_n$ that converge uniformly to $g$. As $g$ is only Hölder, we need a clever method that makes the comparison effective and leads to a Gronwall argument. The comparison must also use ellipticity of $g$.
Previously, \cite{han2025exponential} used the generalized coupling method of \cite{kulik2020well} in the setting. It consists of adding a strong negative drift to the SHE, then combine a pathwise estimate from this strong contraction together with a Girsanov estimate that uses the ellipticity of $g$. The ultimate threshold from this method is, again, $\beta>\frac{3}{4}$. We note that all the contraction estimates and Girsanov estimates of $\|U-V^n\|$ in \cite{han2025exponential} use the distance on $L^2(\mathbb{T})$.

For this type of Hölder $g$ SHEs, the use of $L^2$ distance is fairly natural since the pointwise diffusion coefficient $g$ can only yield a good $L^2$ estimate as we subtract $g(U)-g(V^n)$ pointwise for the candidate solution $U$ and the approximate solution $V^n$. However, it may eventually come up to one's mind that the $L^2$-distance may be suboptimal for this problem, and we can consider the negative Sobolev space $H^{-s}$ for some $s>1/2$. We can instead measure the closeness of $U$ and $V^n$ in this much weaker distance $H^{-s}$ that strongly suppresses higher Fourier modes.

This change of metric may seem a bit cheating at first because the solutions $U,V^n$ nonetheless lie in $L^2(\mathbb{T})$, but in infinite dimensions this change of distance turns out really helpful. One particular reason is that in $H^{-s}$, we can apply Itô formula for the solution since the coefficients are now Hilbert-Schmidt in this norm. This enables a variety of careful estimates that are otherwise unavailable in the $L^2(\mathbb{T})$ metric. Meanwhile, proving weak uniqueness via convergence in the $H^{-s}$ metric is already sufficient for uniqueness in $L^2(\mathbb{T})$.

Then how can we make use of ellipticity of $g$? The previous method of \cite{han2025exponential} uses Girsanov transform, but here we prefer not to use it because Girsanov typically leads to $L^2(\mathbb{T})$ estimates and Girsanov cannot distinguish the signs of plus or minus. Instead, applying Itô to a suitable $C^2$ function can lead to delicate Lyapunov type estimates. Rather, for ellipticity, we are inspired by a reflection coupling method in \cite{wang2015asymptotic} (see also \cite{lindvall1986coupling}, \cite{da2005coupling} for finite-dimensional setting), which splits the noise $gdW$ into two components: one component supplies a scalar non-degenerate white noise while the other component carries the difference $g(U)-g(V^n)$. Then we do a reflection by splitting the first white noise into two more white noises, and the coefficients of them are tailor-made to be a reflection in the direction of $U-V^n$ and at a scale depending on $\|U-V^n\|_{H^{-s}}$. The main point is that when $\|U-V^n\|_{H^{-s}}$ is not too small, the reflection creates a much larger noise in the direction of $U-V^n$, but when tested against a concave test function $f_\delta(r)$, the term in Itô formula involving second derivatives is thus negative. This constitutes a delicate Gronwall argument for the $H^{-s}$-norm of $U-V^n$, and hence leads to uniqueness.
The threshold $\beta>\frac{2}{3}$, which we later discuss, is precisely chosen for this Gronwall argument to close.

Throughout the proof, we are carefully switching between the three Sobolev spaces $H^{-s}$, $L^2(\mathbb{T})$ and also $H^{1-s}$, and we frequently use interpolation estimates between them. Although the reflection coupling method has many precedents in stochastic analysis, the method is seldom used in SHE uniqueness proofs. In particular, the implementation in the current setting involving different Sobolev spaces, and doing reflection in the $L^2$ space at a scale depending on the $H^{-s}$ norm, appears to be fresh. The Itô computations in $H^{-s}$ are also delicate. Using the energy estimate in Lemma \ref{lemma4.237}, the solution is $\ell_t^2H^{1-s}$-valued, and then the interpolation formula \ref{4.3} transfers the $\ell_t^\infty H^{-s}$-estimate to a time-integrated $L^2(\mathbb{T})$-estimate. Such an integrated estimate, unfortunately, cannot be deduced from Girsanov transform, so we used the (slightly weaker) reflection coupling technique and Itô formula.

\subsubsection{Explanation for the current threshold}
We now dig deeper into the proof details and show how the threshold $\frac{2}{3}$ arises. We conjecture that $\frac{2}{3}$ is not the sharp threshold for weak uniqueness, as no counterexamples are known for non-degenerate $g$ for any $\beta>0$. The two natural thresholds beyond $\frac{2}{3}$ are thus $\beta>\frac{1}{2}$ and $\beta>0$. However, $\beta>\frac{2}{3}$ still appears to be the threshold for current coupling-based proof methods unless significant on-diagonal functional analytic mechanisms can be developed.

The obstruction at $\frac{2}{3}$ appears as we optimize over three strict
inequalities. First, to apply Itô calculus in negative Sobolev space $H^{-s}$, the white noise
forces \begin{equation}s>\frac12\end{equation}
in order that \(\operatorname{Tr}(I-\Delta)^{-s}<\infty\).

Let $\delta_n=\max\{\varepsilon_n,n^{-1}\}$ where $\|g_n-g\|_\infty=\epsilon_n\leq \delta_n$, and let $f_{\delta_n}(r)=(r^2+\delta_n^2)^{p/2}$ for some $p\in(0,1)$, We apply Itô calculus to $f_{\delta_n}(\|U(t)-V^n(t)\|_{H^{-s}})$ as a function in $t$.
We estimate the drift \eqref{4.22} after Itô calculus separately in two regimes via different ways: an inner regime (Lemma \ref{lemma5.1}) where the $H^{-s}$-norm of $U-V^n$ is very small or comparable to the coefficient approximation scale, and an outer regime (Lemma \ref{lemma6.1}) where this $H^{-s}$- norm is much larger so that the reflection effect really takes place.

Second, in the inner regime, the Sobolev space interpolation inequality \eqref{4.3}, combined with the use of Young's inequality \eqref{5.11}, \eqref{5.12},
yields a small negative Sobolev norm with the exponent $p-2+2q$ where $q=q(s,\beta)$ is
\begin{equation}
q(s,\beta)=\frac{\beta(1-s)}{1-\beta s},
\label{11.2}
\end{equation}
and this exponent $p-2+2q(s,\beta)$ must be positive because the norm itself is very small.

Third, the reflected distance must use a concave power \(0<p<1\) to guarantee a negative term in \eqref{6.5} when applying Itô to the non-degenerate noise, while
the inner error vanishes only if

\begin{equation}
\gamma=p-2+2q>0,
\qquad\text{equivalently}\qquad p>2-2q.
\label{11.3}
\end{equation}

Such a \(p\) exists exactly when \(q>1/2\).  Since \(q\) decreases with
\(s\), its best possible value is approached as \(s\downarrow1/2\):

\begin{equation}
\sup_{s\in(1/2,1)}q(s,\beta)
=\frac{\beta}{2-\beta}.
\label{11.4}
\end{equation}

The condition in \eqref{11.3} can therefore hold exactly when
\(\beta/(2-\beta)>1/2\), namely \(\beta>2/3\).  At \(\beta=2/3\) all
three choices hit their forbidden endpoints simultaneously.  At
\(\beta=1/2\), the best value is only \(q=1/3\), which would require
\(p>4/3\), directly contradicting the concavity requirement \(p<1\).

Thus optimizing this calculation cannot reach below \(2/3\).  To reach
\(1/2+\varepsilon\), the main bottleneck is the inner regime (Lemma \ref{lemma5.1}) where one would need a genuinely new gain before Young's
inequality \eqref{5.11}, \eqref{5.12}—this is perhaps a space-time cancellation that replaces the obvious bound \eqref{4.6}
\(\|F\|_2^2\lesssim Z^\beta\) by an integrated estimate with an extra
power.  This is difficult because there is no obvious pointwise
sign cancellation.  Reflection controls the radial additive direction $U-V^n$ but
it does not by itself improve this Hölder trace.  This explains both why
the \(2/3\) threshold is already challenging and why the conjectural
\(1/2+\varepsilon\) threshold remains substantially harder.

\section{Preliminaries: notations and noise decomposition}

We begin with some formal definitions of the solution class and the simplest reflection principle.

\subsection{Notation and solution class} Throughout the paper, all Hilbert spaces are real, and we denote by
\(H=L^2(\mathbb T;\mathbb R)\). We fix a countable basis $$e_0=1,\qquad
e_k=\sqrt2\cos(2\pi k\cdot),\qquad
e_{-k}=\sqrt2\sin(2\pi k\cdot),\quad k\ge1.$$ Let

\begin{equation}
A=I-\Delta,
\qquad
\|z\|_{H^a}=\|A^{a/2}z\|_2,
\qquad a\in\mathbb R,
\label{1.3}
\end{equation}
and denote the heat semigroup by \(S_t=e^{t\Delta}\). For any $\alpha\in\mathbb{R}$, the Sobolev space $H^\alpha$ is defined as the closure of $C^\infty(\mathbb{T})$ in the space of distributions on $\mathbb{T}$, under the norm $z\in C^\infty(\mathbb{T})\mapsto \|A^{\alpha/2}z\|_2$. Note that this is the inhomogeneous Sobolev space so that for any $a<0$, $\|z\|_{H^a}=0$ implies $z=0$ in $H$. 

We normalize the
Lebesgue measure of \(\mathbb T\) to be $1$.  A cylindrical Wiener process
on \(H\) is denoted by \(W\).

We introduce a notion of multiplicative operator $M_h$ which shows up the diffusion coefficient of SHE \eqref{1.1}. When $h\in L^\infty(\mathbb{T})$, then $M_h:H\to H,M_h(\phi)=h\phi$ is well-defined operator on $H$. But as the diffusion coefficient $g$ has linear growth, we need to extend the definition of $M_h$ to all $h\in H$.
For $s>1/2$ and $h\in H$, since $L^1(\mathbb{T})$ embeds into $H^{-s}(\mathbb{T})$. We write
\[
\iota_sM_h\in\mathcal L(H,H^{-s})
\]
for the map $\phi\in H\mapsto h\phi\in L^1(\mathbb{T})\subset H^{-s}(\mathbb{T})$. Indeed, we have the estimate
$$
\|\iota_sM_h\phi\|_{H^{-s}}\leq C_s\|h\|_H\|\phi\|_H.
$$

Moreover, for any $t>0$, $S_t\iota_sM_h$ is a (Hilbert-Schmidt) linear operator on $H$ whenever $h\in H,t>0$ via the heat kernel bound later proved in \eqref{3.8f}. We shall always use $S_tM_h$ as an abbreviation for $S_t\iota_sM_h$ throughout this paper.

In this paper, a weak mild solution of \eqref{1.1} consists of a filtered
probability space $(\Omega,\mathbb{P},\mathcal{F})$, a cylindrical Wiener process \(W\), and a predictable
\(H\)-valued process \(U\) such that
\begin{equation}
U(t)=S_tu(0)+
\int_0^t S_{t-r}M_{g(U(r))}\,dW_r,
\qquad t\in[0,T],
\label{1.4}
\end{equation}
and
\begin{equation}
U\in C([0,T];H)
\quad\text{a.s.},
\qquad
\mathbb E\sup_{t\le T}\|U_t\|_2^2<\infty.
\label{1.5}
\end{equation}

Since the coefficient $g$ is Hölder continuous, the existence of a weak mild solution to \eqref{1.1} follows from the standard compactness argument in \cite{gkatarek1994weak}. Furthermore, for such $g$ with linear growth,  the standard
stochastic-convolution estimates (see Lemma \ref{lemma4.237}) give \eqref{1.5} and that, for every
\(s>1/2\),

\begin{equation}
U\in L^2(\Omega;C([0,T];H^{-s}))
\cap L^2(\Omega\times[0,T];H^{1-s}).
\label{1.6}
\end{equation}Related estimates are presented in detail in Lemma \ref{lemma4.237}.

\subsection{Two-noise representation}\label{sections2}
The main idea of this section is that we can split the diffusion term $g(u)dW$ into two independent diffusion terms, one being a multiple of the standard white noise and the other contains the main Hölder irregular part. This construction carries the main idea of the reflection coupling principle in \cite{wang2015asymptotic}, \cite{lindvall1986coupling} and \cite{da2005coupling}. We begin with a split of the coefficient $g(u)$.

Fix a constant $$m[g]\in(0,\frac{c_g}{4})$$ and define

\begin{equation}
h(z)=\sqrt{g(z)^2-m[g]^2}.
\label{2.1}
\end{equation}

Then \(h\in C^\beta\), and the equation with noise

\begin{equation}
m[g]\,dW^0+h(u)\,dW^1
\label{2.2}
\end{equation}

has the same local covariance \(g(u)^2\) as \eqref{1.1}, provided
\(W^0,W^1\) are two independent white noises.

Indeed, since \(h\ge\sqrt{c_g^2-m[g]^2}>2m[g]>0\), we have $h(z)\geq \frac{1}{2}|g(z)|$ for any $z$, so that for any $a,b\in\mathbb{R}$, we have

\begin{equation}
|h(a)-h(b)|
=\frac{|g(a)^2-g(b)^2|}{|h(a)+h(b)|}
\le {[g]_{C^\beta}}|a-b|^\beta\frac{|g(a)+g(b)|}{|h(a)+h(b)|}\le C_{\text{Hol,g}}|a-b|^\beta,
\label{2.2a}
\end{equation}
where $C_{\text{Hol,g}}>0$ is a constant depending only on $[g]_{C^\beta}$ and $c_g$.
Thus the new coefficient $h$ we constructed has the same Hölder index $\beta$.

The following basic lemma in stochastic analysis allows us to split the randomness in the white noise into two independent components. (For finite dimensional Brownian motion the same split follows from Lévy's characterization of Brownian motion.) This is the starting point of our reflection coupling argument, where we reflect only part of the white noise.
\begin{lemma}[Orthogonal split of a white noise]\label{lemmas2.1}

Let \((W,W')\) be two independent cylindrical Wiener processes on \(H\),
and let \(a_t,b_t\) be bounded predictable processes
satisfying \(a_t(x)^2+b_t(x)^2=1\) for all $x\in\mathbb{T}$ (so that both processes lie in $L^\infty(\mathbb{T})$ a.s.) Define

\begin{equation}
\begin{aligned}
dB^0_t&=a_t\,dW_t+b_t\,dW'_t,\\
dB^1_t&=b_t\,dW_t-a_t\,dW'_t.
\end{aligned}
\label{2.2b}
\end{equation}
Then \((B^0,B^1)\) is a pair of independent cylindrical Wiener
processes in the enlarged filtration.

\end{lemma}
\begin{proof}

Take two test directions \(\phi,\psi\in H\), then independence of $W_t,W_t'$ implies

\begin{equation}
\begin{aligned}
d[B^0(\phi),B^0(\psi)]_t
&=\langle a_t\phi,a_t\psi\rangle dt
+\langle b_t\phi,b_t\psi\rangle dt
=\langle\phi,\psi\rangle dt,\\
d[B^1(\phi),B^1(\psi)]_t
&=\langle\phi,\psi\rangle dt,\\
d[B^0(\phi),B^1(\psi)]_t
&=\langle a_t\phi,b_t\psi\rangle dt
-\langle b_t\phi,a_t\psi\rangle dt=0.
\end{aligned}
\label{2.2c}
\end{equation}
For every finite family of tests \(\phi_1,\ldots,\phi_k\in H\), the resulting
\(2k\)-dimensional continuous local martingale has the specified
Brownian covariance.  Lévy's characterization then makes it a standard
Brownian motion with the indicated covariance.  This finite-dimensional consistency is precisely the defining property of a cylindrical Brownian motion, and the zero
cross-brackets give independence of $B^0$ and $B^1$.  \end{proof}

This representation can be installed on \textbf{any prescribed weak
solution}.  Indeed, let \((u,W)\) be such a solution to \eqref{1.1}, we enlarge the space
by an independent white noise \(W'\), and put, pointwise and predictably,

\begin{equation}
\begin{aligned}
dW^0&=\frac {m[g]}{g(u)}\,dW+\frac{h(u)}{g(u)}\,dW',\\
dW^1&=\frac{h(u)}{g(u)}\,dW-\frac {m[g]}{g(u)}\,dW'.
\end{aligned}
\label{2.3}
\end{equation}

Apply Lemma \ref{lemmas2.1} with the two predictable processes defined by

\begin{equation}
a=\frac {m[g]}{g(u)},
\qquad
b=\frac{h(u)}{g(u)}.
\label{2.3a}
\end{equation}

The identity \(a^2+b^2=1\) shows that \(W^0,W^1\) are independent white
noises in the enlarged filtration.  Moreover,

\begin{equation}
m[g]\,dW^0+h(u)\,dW^1=g(u)\,dW,
\label{2.4}
\end{equation}
so the given solution $(u,W)$ has not been changed.

\section{The reflection coupling with respect to an arbitrary solution}\label{sections3}
Since \(g\) is continuous and \(|g|\ge c_g\), it has a constant sign.
Replacing \(W\) by \(-W\) if necessary, we assume henceforth that
\(g\ge c_g\).

We will prove weak uniqueness via an approximation argument. Choose smooth, globally Lipschitz functions $g_n\in C^\infty(\mathbb{R})$ satisfying that
\begin{equation}
g_n\longrightarrow g\quad\hbox{uniformly on $\mathbb{R}$},
\qquad
c_g/2\le g_n, \quad [g_n]_{C^\beta}\leq 2[g]_{C^\beta}
\label{3.1}
\end{equation}
(a standard mollifier on $\mathbb{R}$ suffices to get $g_n$: let $\rho$ be a smooth compactly supported function with integral one and define $g_\epsilon=\rho_\epsilon*g$, then $|g_\epsilon(x)-g(x)\leq [g]_{C^\beta}\epsilon^\beta |\int_\mathbb{R} \rho(z)|z|^\beta dz|$. Moreover, since $\int \rho_\epsilon'=0$, then from $g_\epsilon'(x)=\int \rho_\epsilon'(y)(g(x-y)-g(x))dy$ we see that $|g_\epsilon'(x)|$ has an $\epsilon$-dependent bound for all $x$),
and define
\begin{equation}
h_n(z)=\sqrt{g_n(z)^2-m[g]^2},
\qquad
\varepsilon_n:=\|h_n-h\|_\infty\longrightarrow0.
\label{3.2}
\end{equation}

Because \(g_n\ge c_g/2\geq 2m[g]\), each \(h_n\) is smooth and globally
Lipschitz.  The map $x\mapsto \sqrt{x^2-m[g]^2}$ is uniformly Lipschitz for 
\(x\in[c_g/2,\infty)\), so uniform convergence of \(g_n\) also gives
the asserted uniform convergence of \(h_n\). Moreover, by computations in \eqref{2.2a}, $[h_n]_{C^\beta}<\infty$ is uniformly bounded in $n$.

The following $g_n$-equation, having Lipschitz coefficients
\begin{equation}
\partial_tu_n=\Delta u_n+g_n(u_n)\,\dot W,
\qquad u_n(0)=u_0\in L^2(\mathbb{T}),
\label{unequation}\end{equation}has a unique strong solution. Hence the solution $u_n$ has a canonical
law, denoted by \(P_n\).

Fix an arbitrary weak solution \(U\) of the original \(g\)-equation \eqref{1.1} and install
the noises \(W^0,W^1\) as in Section \ref{sections2}.  We then adjoin a third independent
white noise \(\overline W^0\) for a future technical reason.  We let \(V^n\) be the process of strong solutions associated to the $g_n$-equation $u_n$ in \eqref{unequation} with the canonical
law \(P_n\), defined on this enlarged probability space supporting $W^0,W^1$.

Denote by, for some currently unfixed $s>0$, the following norms and rescaling of $D\in H$:

\begin{equation}
A=I-\Delta,
\qquad
r_s(D):=\|D\|_{H^{-s}}=\|A^{-s/2}D\|_2,
\qquad
e(D)=\frac D{\|D\|_2},\qquad D\ne0\in H,
\label{3.3}
\end{equation}since $D\in H,$ then $D\in H^{-s}$ as $s>0$. Here $\|D\|_2=\|D\|_{L^2(\mathbb{T})}$ is the $L^2$-norm. If $D=0$, then we simply define $e(D)=0$.

The exponents \(s\) will be fixed in \eqref{5.1}. Take another constant $p\in(0,1)$ which will be later fixed in \eqref{5.3}. Given a (small) approximation scale
\(\delta>0\), choose \(K_0>2/\sqrt{1-p}\) and we choose a smooth angle function
\(\theta_\delta:[0,\infty)\to[0,\pi]\) such that

\begin{equation}
\theta_\delta(r)=0\quad(r\le K_0\delta),
\qquad
\theta_\delta(r)=\pi\quad(r\ge2K_0\delta).
\label{3.4}
\end{equation}
and such that for some universal constant $C_0>0$,
$$\|\theta_\delta'\|_\infty\le C_0/\delta$$(such a function $\theta_\delta$ can be chosen by taking a smooth transition profile and then rescale it).  We define the following two mappings $K_\delta,S_\delta:H\to \mathcal{L}(H)$, with
\(e\otimes e=0\) when \(D=0\),

\begin{equation}
\begin{aligned}
K_\delta(D)
&=I+(\cos\theta_\delta(r_s(D))-1)e(D)\otimes e(D),\\
S_\delta(D)
&=\sin\theta_\delta(r_s(D))e(D)\otimes e(D).
\end{aligned}
\label{3.5}
\end{equation}

Then we can verify that
\begin{lemma}\label{equalitytwomaps}The maps $K_\delta,S_\delta$ satisfy that
\begin{equation}
K_\delta K_\delta^*+S_\delta S_\delta^*=I.
\label{3.6}
\end{equation}
\end{lemma}
\begin{proof}
    Denote by $P_D=e(D)\otimes e(D)$, which is a rank-one operator via $P_Dx=\langle x,e(D)\rangle e(D)$. Then $P_D$ is a projection and $P_D^2=P_D=P_D^*$. Write $c_1=\cos\theta_\delta(r_s(D))$ and $c_2=\sin\theta_\delta(r_s(D))$, then $c_1^2+c_2^2=1$. And consequently
    $$
K_\delta K_\delta^*+S_\delta S_\delta^*=[I+(c_1-1)P_D]^2+c_2^2P_D^2=I.
    $$
\end{proof}
We define an auxiliary process with respect to this noise process after orthogonal transform, which we eventually check is indeed the solution process to $u_n$ with law $P_n$. Let \(V^n\) solve the following SHE

\begin{equation}
\begin{aligned}
dV^n(t)={}&\Delta V^n(t)\,dt
+m[g]K_\delta(U(t)-V^n(t))\,dW^0
+m[g]S_\delta(U(t)-V^n(t))\,d\overline W^0\\
&+h_n(V^n(t))\,dW^1,
\qquad V^n(0)=U(0).
\end{aligned}
\label{3.7}
\end{equation}

We first check that the coefficients in the SHE $V^n$ satisfy bounded and Lipschitz properties. The crucial point here is that we only reflect in one direction $U-V^n$, so the rank-1 update coefficient maps $K_\delta,S_\delta$ are indeed Lipschitz continuous in Hilbert-Schmidt norm.

\begin{lemma}[Lipschitz bounds on the smoothed reflection SHEs]\label{lemmas3.1}
For every fixed \(\delta>0\),

\begin{equation}
D\to K_\delta(D)-I,
\qquad
D\to S_\delta(D)
\label{3.8}
\end{equation}
are globally Lipschitz maps from \(H\) to the Hilbert--Schmidt
operators on \(H\).  More precisely,

\begin{equation}
\|K_\delta(D)-K_\delta(E)\|_{\mathrm{HS}}
+\|S_\delta(D)-S_\delta(E)\|_{\mathrm{HS}}
\le C_{K_0}\delta^{-1}\|D-E\|_H,
\label{3.8a}
\end{equation}
where $C_{K_0}>0$ only depends on $K_0$.
Furthermore,

\begin{equation}
\|K_\delta(D)\|_{\mathrm{op}}\le1,
\qquad
\|S_\delta(D)\|_{\mathrm{op}}\le1,
\qquad
K_\delta K_\delta^*+S_\delta S_\delta^*=I.
\label{3.8b}
\end{equation}

\end{lemma}
Notice carefully
that \(K_\delta(D)\) itself is not Hilbert--Schmidt: it is the identity
plus a rank-one operator.  Only its differences are Hilbert--Schmidt.
\begin{proof}

The map \(r_s(D)=\|A^{-s/2}D\|_2\) is 1-Lipschitz on \(H\) since $s>0$.  For
\(D,E\ne0\), recall that \(P_D=e(D)\otimes e(D)\) is a rank-one operator, then we take the difference:
$$
\|P_D-P_E\|_{HS}\leq \|e(D)\otimes (e(D)-e(E))\|_{HS}+\|e(E)\otimes (e(D)-e(E))\|_{HS}\leq 2\|e_D-e_E\|_2,
$$using that $\|e_D\|_2=\|e_E\|_2=1$. Then by triangle inequality again,

\begin{equation}
\|e_D-e_E\|_2
\le \left\|\frac {D-E}{\|D\|_2}\right\|_2+\|E\|_2\left\|\frac{1}{\|E\|_2}-\frac{1}{\|D\|_2}\right\|_2
\le \frac{2\|D-E\|_2}{\|D\|_2}.
\label{3.8c}
\end{equation}

Consider \(G(D)=\zeta_\delta(r_s(D))P_D\), where \(\zeta_\delta=0\)
on \([0,K_0\delta]\), \(\|\zeta_\delta\|_\infty\le2\), and
\(\operatorname{Lip}(\zeta_\delta)\le C/\delta\).  If both points are
in the active region, so that $\|D\|_{H^{-s}}\geq K_0\delta,\|E\|_{H^{-s}}\geq K_0\delta$, then
\(\|D\|_2,\|E\|_2\ge K_0\delta\), and \eqref{3.8c}, together with the
Lipschitz bound on \(\zeta_\delta\), proves \eqref{3.8a}.  If, say,
\(r_s(E)\le K_0\delta<r_s(D)\), use \(G(E)=0\) and thus

\begin{equation}
\|G(D)\|_{\mathrm{HS}}
=|\zeta_\delta(r_s(D))-\zeta_\delta(r_s(E))|
\le C_{K_0}\delta^{-1}\|D-E\|_2.
\label{3.8d}
\end{equation}
This proves the assertion \eqref{3.8a} on $K_\delta,S_\delta$ with the choice
\(\zeta_\delta=\cos\theta_\delta-1\) and with
\(\zeta_\delta=\sin\theta_\delta\).  On the line spanned by \(D\),
the two operators have eigenvalues \(\cos\theta_\delta(r_s(D))\) and
\(\sin\theta_\delta(r_s(D))\); on its orthogonal complement they have
eigenvalues \(1\) and \(0\).  This proves \eqref{3.8b}.  
\end{proof}

It is then standard to verify that the SPDE \eqref{3.7} has a unique mild solution $V^n$:

\begin{lemma}[Uniqueness of solutions $V^n$]\label{lemmas3.3}
For each fixed \(n,\delta,T\), the equation \eqref{3.7} has a unique predictable mild
solution \(V^n\in C([0,T];H)\), up to indistinguishability, and

\begin{equation}
\sup_{t\le T}\|V_t^n\|_2^2<\infty,\quad  a.s.
\label{3.8e}
\end{equation}

\end{lemma}
Note that we are not yet proving an estimate uniform in $n$ and $\delta$. The a-priori estimate $V^n\in C([0,T];H)$ will be done again for the process $U$ in Lemma \ref{lemma4.237}.

\begin{proof} 
We regard the solution $U$ as a prescribed predictable process. 

On the one-dimensional torus, for $0<t\le T$ and $f\in L^2(\mathbb T)$,
\begin{equation}
\|S_tM_f\|_{\mathrm{HS}}^2
=
\int_{\mathbb T}p_{2t}(x,x)|f(x)|^2\,dx
\le Ct^{-1/2}\|f\|_2^2.
\label{3.8f}
\end{equation}
In particular,
\[
    \|S_t\|_{\mathrm{HS}}^2\le Ct^{-1/2},
\]
which is integrable on $[0,T]$. Since $h_n$ is globally Lipschitz,
\begin{equation}
\|S_tM_{h_n(v)-h_n(w)}\|_{\mathrm{HS}}^2
\le
C_Tt^{-1/2}[h_n]_{\mathrm{Lip}}^2\|v-w\|_2^2.
\label{3.8g}
\end{equation}

By Lemma \ref{lemmas3.1} and using
$\|S_t\|_{\mathrm{op}}\le 1$,
\begin{equation}
\begin{aligned}
&\|S_t(K_\delta(D)-K_\delta(E))\|_{\mathrm{HS}}
 +\|S_t(S_\delta(D)-S_\delta(E))\|_{\mathrm{HS}}  \\
&\hspace{3cm}
\le C\delta^{-1}\|D-E\|_2.
\end{aligned}
\label{3.8h}
\end{equation}
Likewise, using also Lemma \ref{lemmas3.1},
\begin{equation}
\begin{aligned}
\|S_tK_\delta(D)\|_{\mathrm{HS}}
&\le
\|S_t\|_{\mathrm{HS}}\|K_\delta(D)\|_{\mathrm{op}}
\le \|S_t\|_{\mathrm{HS}},\\
\|S_tS_\delta(D)\|_{\mathrm{HS}}
&\le
\|S_t\|_{\mathrm{HS}}\|S_\delta(D)\|_{\mathrm{op}}
\le \|S_t\|_{\mathrm{HS}}.
\end{aligned}
\label{3.8i}
\end{equation}

Then a standard Picard iteration argument in the sense of \cite[Section 7]{da2014stochastic} proves the pathwise uniqueness of $V^n$. The a-priori estimate $V^n\in C([0,T];H)$ follows from \cite[Section 5]{da2014stochastic}, and see also Lemma \ref{lemma4.237} where the same argument is used again for energy estimates.
\end{proof}

\begin{lemma}[Preservation of the smooth marginal]\label{preservations}

The solution $V^n$ to \eqref{3.7} satisfies

\begin{equation}\operatorname{Law}(V^n)=P_n,
\label{3.8jj}
\end{equation}
and this marginal law is independent of the candidate solution \(U\).

\end{lemma}
\begin{proof}

Denote by $\widetilde W^{0,n}$ the noise process associated with $V^n$:

\begin{equation}
d\widetilde W^{0,n}
=K_\delta(U-V^n)dW^0+S_\delta(U-V^n)d\overline W^0.
\label{3.9}
\end{equation}

For \(\phi,\psi\in H\), by Lemma \ref{equalitytwomaps},

\begin{equation}
\begin{aligned}
d[\widetilde W^{0,n}(\phi),
\widetilde W^{0,n}(\psi)]_t
&=\langle\phi,(K_\delta K_\delta^*
+S_\delta S_\delta^*)\psi\rangle dt\\
&=\langle\phi,\psi\rangle dt.
\end{aligned}
\label{3.9a}
\end{equation}
Its cross-bracket with \(W^1\) is zero because \(W^0,\overline W^0\)
have zero cross-variation with \(W^1\).  This is still true even though
the predictable integrands depend on \(W^1\).  For each finite family
of spatial test functions, the joint process is therefore a continuous
local martingale with deterministic Brownian covariance.  The
finite-dimensional Lévy characterization proves that
\((\widetilde W^{0,n},W^1)\) is a pair of independent cylindrical white
noises in the full joint filtration.  Therefore, the marginal equation for
\(V^n\) has covariance
$
m[g]^2+h_n(V^n)^2=g_n(V^n)^2.
$ Explicitly,
$
d\widehat W^n
=\frac{m[g]}{g_n(V^n)}d\widetilde W^{0,n}
+\frac{h_n(V^n)}{g_n(V^n)}dW^1
$
is a cylindrical white noise by the same finite-dimensional bracket
calculation.

Thus \(V^n\) is a weak solution of the globally Lipschitz
\(g_n\)-equation \eqref{unequation}.  Pathwise uniqueness for \eqref{unequation} gives uniqueness in law and
hence the equality in law of \eqref{3.8jj}.
\end{proof}

\section{Energy estimates, Itô formula and the drift}
We will do a contraction estimate between the candidate weak solution $U$ and the  approximation process $V^n$. We write their differences as follows:
(the value $s\in(0,1)$ will be fixed later)
\begin{equation}
D(t)=U(t)-V^n(t),
\qquad
r_s(t)=\|D(t)\|_{H^{-s}},
\qquad
Z(t)=\|D(t)\|_2^2,
\qquad
\mathcal A_s(t)=\|D(t)\|_{H^{1-s}}^2,
\label{4.1}
\end{equation}
where $D(t),U(t),V^n(t)$ are the time-$t$ values of the process $D,U,V^n$.

 Since the evolution equation has better regularity in the negative Sobolev space $H^{-s}$, we rewrite the equation and do the comparison there. Let \(\iota:L^2\to H^{-s}\) be the canonical
inclusion, then the difference equation of $U$ and $V^n$ in \(H^{-s}\) can be written as
\begin{equation}
dD(t)=\Delta D(t)\,dt+B_0(D(t))dW^0+B_\perp(D(t))d\overline W^0
+B_{F(t)}dW^1,
\label{4.1a}
\end{equation}
where
\begin{equation}\begin{aligned}
&B_0(D(t))=m[g]\iota(I-K_\delta(D(t))),
\qquad B_\perp(D(t))=-m[g]\iota S_\delta(D(t)),
\\& B_{F(t)}=\iota M_{F(t)},
\qquad F(t)=h(U(t))-h_n(V^n(t)).\end{aligned}
\label{4.1b}
\end{equation}
Here the coefficients $h,h_n$ have linear growth under the hypotheses of Theorem \ref{theorem1.11} and $U(t),V^n(t)\in H$, so $F\in H$ and \(\iota_sM_F\) is well-defined from $H$ to $H^{-s}$.

We prove the following regularity estimate for $B_F$ in the space $H^{-s}$:
\begin{lemma}[$L^2$ difference of diffusion coefficient and interpolation estimates]\label{lemmas4.111}

For each \(s>1/2\), we can find a constant $C_{s,g}>0$ depending only on $s$ and the function $g$ such that 

\begin{equation}
\|B_{F(t)}\|_{\mathcal{L}_2(H,H^{-s})}^2
=\|A^{-s/2}M_{F(t)}\|_{\mathrm{HS}}^2
\le C_{s,g} Z(t)^\beta+C_{s,g}\varepsilon_n^2,
\label{4.2}
\end{equation}
where $\mathcal{L}_2(H,H^{-s})$ is the space of Hilbert-Schmidt operators from $H$ to $H^{-s}$
and

\begin{equation}
Z(t)\le r_s^{2(1-s)}(t)\mathcal A_s^s(t),\quad \forall t\geq 0.
\label{4.3}
\end{equation}

\end{lemma}
\begin{proof}
Let $\{e_k\}_{k\in\mathbb{Z}}$ be a real orthonormal basis with $e_0=1$, $e_k(\cdot)=\sqrt2\cos(2\pi k\cdot),k\in\mathbb{N}_+$ and $e_k(\cdot)=\sqrt2\sin(2\pi k\cdot),k\in\mathbb{N}_{-}$ and let \(K_s(x,y)\) be the kernel
of \(Q:=A^{-s}\).  Parseval and Tonelli's theorem give

\begin{equation}
\begin{aligned}
\|B_{F(t)}\|^2_{\mathcal{L}_2(H,H^{-s})}=\sum_k
\|Q^{1/2}M_{F(t)}e_k\|_2^2
&=\sum_k\|Q^{1/2}(F(t)e_k)\|_2^2\\
&=\int_{\mathbb T}K_s(x,x)|F(t)(x)|^2dx.
\end{aligned}
\label{4.4}
\end{equation}

On the torus,

\begin{equation}
K_s(x,x)=\sum_{k\in\mathbb Z}
\big(1+4\pi^2k^2\big)^{-s}=:\kappa_s<\infty
\quad\Longleftrightarrow\quad s>\frac12.
\label{4.5}
\end{equation}
(There is a more transparent way of deriving this on $\mathbb{T}$: set the complex Fourier basis $\widehat{e}_k(x)=e^{2\pi ikx}$ and write $\widehat{F}(j)$ the Fourier coefficient of $F$, then $\widehat{F\widehat{e}_k}(j)=\widehat{F}(j-k)$ and $\sum_k\|A^{-s/2}(F\widehat{e}_k)\|_2^2=\sum_k\sum_j\lambda_j^{-s}|\widehat{F}(j-k)|^2$).
Since \(h\in C^\beta\) and
\(\|h-h_n\|_\infty\leq\varepsilon_n\), the norm of $F(t)$ can be bounded by

\begin{equation}
\begin{aligned}
\|F(t)\|_2^2
&\le2\|h(U(t))-h(V^n(t))\|_2^2+2\varepsilon_n^2\\
&\le C_{\text{Hol,g}}\int_{\mathbb T}|D(t)(x)|^{2\beta}dx
+2\varepsilon_n^2\\
&\le C_{\text{Hol,g}}\|D(t)\|_2^{2\beta}+2\varepsilon_n^2.
\end{aligned}
\label{4.6}
\end{equation}
The last step uses \(2\beta<2\) and that $\mathbb{T}$ has unit volume.

Finally, with
\(\lambda_k=1+4\pi^2k^2\), let $\{D_k\}$ be the coefficient of $D(t)$ with respect to the complex Fourier basis $\{\widehat{e}_k\}$, then Hölder's inequality for the sum $\|D(t)\|_2^2=\sum_k|D_k|^2$ gives

\begin{equation}
\begin{aligned}
Z(t)
&=\sum_k
\big(\lambda_k^{-s}|D_k|^2\big)^{1-s}
\big(\lambda_k^{1-s}|D_k|^2\big)^s\\
&\le
\left(\sum_k\lambda_k^{-s}|D_k|^2\right)^{1-s}
\left(\sum_k\lambda_k^{1-s}|D_k|^2\right)^s,
\end{aligned}
\label{4.7}
\end{equation}
which gives estimate \eqref{4.3} by the definition of $r_s(t),\mathcal{A}_s(t)$.
\end{proof}

We next show that the difference process $D$ lies in the Sobolev space $H^{1-s}$ and has the following spatial-temporal continuity, so that Itô's formula can be directly applied:
\begin{lemma}[Energy estimates and Itô identity]\label{lemma4.237}

For every \(T<\infty\) and every $s\in(\frac{1}{2},1)$,

\begin{equation}
D\in L^2(\Omega;C([0,T];H^{-s}))
\cap L^2(\Omega\times[0,T];H^{1-s}),
\label{4.8}
\end{equation}indeed we have the following more standard $L^2$-estimate 
\begin{equation}
D\in L^2(\Omega;C([0,T];H))
.
\label{stronger4.8}
\end{equation}
The variational Itô formula can be formally applied to
\((r_s(t)^2+\delta^2)^{p/2}\) for any $\delta>0,p\in(0,1)$ and we obtain an upper bound \eqref{4.22} on the drift after this application of Itô formula. 

\end{lemma}

The main component of this lemma is twofold: first, we prove the required functional space estimates for the difference process $D$ via deriving energy estimates; and second, we formally compute the drift after applying Itô formula to $D(t)$. The rigorous verification that Itô  can be rigorously applied to $f_\delta(D(t))$ is presented in Lemma \ref{lemmas7.1}.

\begin{proof}
In deriving energy estimates, we make crucial use of the fact that $g$, $g_n$ and $h_n$ have linear growth: since each of them is $\beta$-Hölder continuous, we have that
$$
|g(x)|\leq |g(0)|+[g]_{C^\beta}(1+|x|),\quad x\in\mathbb{R}
,$$
and similarly for $g_n,h_n$.

\textbf{Background: variational Itô formulas for SPDEs and Gelfand triples.}
We shall work with the following Gelfand triple

\begin{equation}
H^{1-s}\subset H^{-s}\subset H^{-1-s},
\label{4.9}
\end{equation}
in the sense that $H^{1-s}$ embeds continuously and densely into $H^{-s}$, and regarding $H^{-s}$ as the central Hilbert space, it also embeds continuously and densely into $H^{-1-s}=(H^{1-s})^*$ where the dual is defined with respect to the norm of $H^{-s}$.

Recall that for any such Gelfand triple $V\subset \widehat{H}\subset V^*$, suppose an adapted process $X$ satisfies, in $V^*$,
\begin{equation}\label{integralformulation}
X_t=X_0+\int_0^t a_rdr+\int_0^tb_rdW_r,
\end{equation}where $X\in L^2(\Omega\times[0,T];V)$ and $a\in L^2(\Omega\times[0,T];V^*)$ and the noise process is $b\in L^2(\Omega\times[0,T];\mathcal{L}_2(\mathbb{U},\widehat{H}))$ \footnote{Here $\mathcal{L}_2(\mathbb{U},\widehat{H})$ is the linear space of Hilbert-Schmidt operator from $\mathbb{U}$ to $\widehat{H}$.} and $\mathbb{U}$ is the noise space, then the variational Itô formula in \cite[Theorem 4.2.5]{PrevotRockner2007} implies that $X$ has a $\widehat{H}$-continuous version and we have
\begin{equation}\label{quadraticitos}
    \begin{aligned}
\|X_t\|_{\widehat{H}}^2&=\|X_0\|_{\widehat{H}}^2+2\int_0^t \langle a_r,X_r\rangle_{V^*,V}dr\\&+\int_0^t \|b_r\|^2_{\mathcal{L}_2(\mathbb{U};\widehat{H})}dr+    2\int_0^t\langle X_r,b_rdW_r\rangle_{\widehat{H}}.    
    \end{aligned}
\end{equation}
For the rest of the proof we take $\mathbb{U}=H\oplus H\oplus H$ and $\widehat{H}=H^{-s}$, $V=H^{1-s}$.

\textbf{Step 1: $L^2$ energy estimates.} We first verify that 
$$U,V^n\in L^2(\Omega;C([0,T];H))
.$$
Since $H$ embeds into $H^{-s}$, this immediately gives $D\in L^2(\Omega;C([0,T];H^{-s})).$

We perform an a priori estimate for any weak mild
solution $U$ and does not use uniqueness. Since $g$ has linear growth, the heat kernel smoothing estimate in \eqref{3.8f} gives,
$$\|S_tM_{g(U(r))}\|_{\mathcal L_2(H)}^2
 \lesssim t^{-1/2}(1+\|U(r)\|_H^2)
$$up to a constant depending only on $[g]_{C^\beta}$. This term $t^{-1/2}$ is integrable on $[0,T]$.
Applying the stochastic factorization formula and the deterministic
factorization estimate
in \cite[Section~5.3, Proposition~5.9 and
Theorem~5.10]{da2014stochastic}, first up to the stopping time
$\tau_R=\inf\{t:\|U_t\|_H\ge R\}$, gives for some $p>4$
\[
 \mathbb E\sup_{t\le T\wedge\tau_R}\|U(t)\|_H^p
 \le C_{p,T,g}(1+\|U(0)\|_H^p),
\]
where the constant is independent of $R$. Letting $R\to\infty$ and
using Fatou's lemma proves
\begin{equation}\label{usigfatou}
 \mathbb E\sup_{t\le T}\|U(t)\|_H^2
 \le C_{T,g}(1+\|U(0)\|_H^2).
\end{equation}
The same estimate holds for $V^n$ in place of $U$, and the leading numerical coefficient $C_{T,g_n}$ can be made uniformly bounded for $g_n$ since $[g_n]_{C^\beta}\leq 2[g]_{C^\beta}$.

Moreover, the standard factorization estimates in \cite[Section 5]{da2014stochastic} give continuity of $U$, so we have $$
 U\in L^2\bigl(\Omega;C([0,T];H)\bigr),
 \qquad
 \mathbb E\|U\|_{C([0,T];H)}^2
 \le C_{T,g}(1+\|u_0\|_H^2),
$$
and the same also holds for $V^n$.

We extract one useful estimate here: following the computation in \eqref{4.4},
we can find $C_{T,g}>0$ depending only on $T,g,s$ such that for all $n$,
\begin{equation}\label{lines8560}
\mathbb{E}{\sup_{ t\leq T}}\|\iota_sM_{h(U(t))}\|_{\mathcal{L}_2(H,H^{-s})}^2\leq C_{T,g},\quad \mathbb{E}{\sup_{ t\leq T}}\|\iota_s M_{h_n(V^n(t))}\|_{\mathcal{L}_2(H,H^{-s})}^2\leq C_{T,g}.
\end{equation}

\textbf{Step 2: Integrability in the norm $H^{1-s}$.}
We then obtain the \(H^{1-s}\) regularity without using variational
Itô.

Denote by 
$$
q_s(u):=\sum_{k\in\mathbb{Z}}\lambda_k^{1-s}e^{-2(\lambda_k-1)u},
$$where $\lambda_k=1+4\pi^2k^2,k\in\mathbb{Z}$ are eigenvalues of $A=I-\Delta$,
then by translational invariance, 
\begin{equation}\label{lines869}
\|A^{(1-s)/2}S_uM_f\|_{HS}^2=q_s(u)\|f\|_H^2,
\end{equation}
and that for any $s>\frac{1}{2}$,
\begin{equation}\label{lines870}
\int_0^T q_s(u)du\leq T+C\sum_{k\neq 0}\frac{\lambda_k^{1-s}}{\lambda_k-1}\leq C_{s,T}<\infty.
\end{equation}

Then Itô's isometry and Tonelli inequality gives, with $f(r):=h(U(r))$ or $f(r):=h_n(V^n(r))$,

\begin{equation}
\begin{aligned}
&\mathbb E\int_0^T\left\|
A^{(1-s)/2}\int_0^tS_{t-r}M_{f(r)}\,dW_r
\right\|_2^2dt\\
&\quad=\mathbb E\int_0^T\int_0^t
\|A^{(1-s)/2}S_{t-r}M_{f(r)}\|_{\mathrm{HS}}^2drdt\\
&\quad\le C_{s,T}C_{g}\mathbb{E}\sup_{t\leq T}(1+\|U(t)\|_H^2+\|V^n(t)\|_H^2)<\infty,
\end{aligned}
\label{4.9a}
\end{equation}
where in the last step we used \eqref{usigfatou}, \eqref{lines869}, \eqref{lines870} and the linear growth of $h,h_n$.

For the coefficients $K_\delta(D(r)),S_\delta(D(r))$, they are already bounded operators on $H$, so a similar but simpler estimate with $M_{f(r)}$ replaced by $K_\delta,S_\delta$ gives similar upper bounds.

The deterministic
term \(S_tu_0\) obeys the same integrated estimate.  Therefore, all the processes
\(U,V^n\), and \(D\) belong to \(L^2(\Omega\times[0,T];H^{1-s})\).

\textbf{Step 3: Validation for the application of variational Itô.}
We may now apply quadratic variational Itô formula \eqref{quadraticitos}. We have the following calculus: recall that $Q=A^{-s}$ and the $H^{-s}$-inner product is $\langle x,y\rangle_{H^{-s}}=\langle Qx,y\rangle_{L^2}$, then 
\begin{equation}
\begin{aligned}\langle& \Delta X,X\rangle_{H^{-1-s},H^{1-s}}=\langle \Delta X,QX\rangle_{H^{-1-s},H^{1+s}}
\\&=-\langle(1-\Delta)^{-s}X,(1-\Delta)X \rangle_{L^2}+\langle(1-\Delta)^{-s}X,X\rangle_{L^2}=-\|X\|_{H^{1-s}}^2+\|X\|_{H^{-s}}^2.
\label{4.11}
\end{aligned}\end{equation} Here as we have verified that $X\in L^2(\Omega\times[0,T];H^{1-s}),$ then $\Delta X\in L^2(\Omega\times[0,T];H^{-1-s})$.
Moreover, the diffusion coefficients in \eqref{4.1a} satisfy
\[
 \mathbb E\int_0^T
 \sum_{j\in\{0,\perp,F\}}
 \|B_j(t)\|_{\mathcal L_2(H,H^{-s})}^2dt<\infty.
\]
Indeed, the reflection coefficients are uniformly Hilbert--Schmidt
from $H$ to $H^{-s}$, while the estimate for $B_F(t)$ is established in \eqref{lines8560}. From the mild formulation \eqref{4.1a}, we can test it by smooth test functions to turn it into an integral equation \eqref{integralformulation} in $H^{-1-s}$ via stochastic Fubini.

Consequently, we checked that the mild formulation \eqref{4.1a} can be written as a variational equation for the Gelfand
triple
\[
 H^{1-s}\subset H^{-s}\subset H^{-1-s}.
\]
The quadratic variational Itô formula can now be applied to
$\|D(t)\|_{H^{-s}}^2$. The localization and integrability needed
below when taking expectations are verified later in
Lemma~\ref{lemmas7.1}.

\textbf{Step 4: Itô applied to function.}
Let \(R_s(t)=r_s(t)^2=\langle QD(t),D(t)\rangle_{L^2}\). Denote by $W^\perp=\overline{W}^0$ and $W^F=W^1$, then the quadratic energy identity applied to \eqref{4.1a} is

\begin{equation}
\begin{aligned}
dR_s(t)={}&2(-\mathcal A_s(t)+r_s(t)^2)dt
+\sum_{j\in\{0,\perp,F\}}
\|B_j\|_{\mathcal{L}_2(H,H^{-s})}^2dt\\
&+2\sum_{j\in\{0,\perp,F\}}
\langle D(t),B_j(t)dW^j\rangle_{H^{-s}}.
\end{aligned}
\label{4.12}
\end{equation}
where we used \eqref{4.11} inside the variational Itô formula, and the quadratic variation satisfies 
$$
d[ R_s]_t=4\sum_{j\in\{0,\perp,F\}}\|B_j(t)^*D(t)\|_{H^{}}^2dt
$$(note that $B_j^*$ maps $H^{-s}$ to $H$),
For fixed \(\delta>0\), $p\in(0,1)$, the scalar function
$$\Phi_\delta(R)=(R+\delta^2)^{p/2}$$ is \(C^2\) on
\([0,\infty)\), and its composition with the squared Hilbert norm has
bounded first and second Fréchet derivatives on \(H^{-s}\).  Hence
ordinary scalar Itô formula can be (formally) applied to \eqref{4.12}. Applying Itô to $\Phi_\delta(R_s(t))$, we get 
$$
d\Phi_\delta(R_s(t))=\Phi_\delta'(R_s(t))dR_s(t)+\frac{1}{2}\Phi_\delta''(R_s(t))d\langle R_s\rangle_t.
$$ The drift $\mathfrak{b}(t)$ of $d\Phi_\delta(R_s(t))$ at time $t$ is then given by 
\begin{equation}\begin{aligned}
\mathfrak{b}(t)=&2\Phi_\delta'(R_s(t))(-\mathcal{A}_s(t)+r_s(t)^2)\\&+\sum_{j\in\{0,\perp,F\}}[\Phi_\delta'(R_s(t))\|B_j(t)\|^2_{\mathcal{L}_2(H,H^{-s})}+2\Phi_\delta''(R_s(t))\|B_j(t)^*D(t)\|_{H^{}}^2].\end{aligned}
\end{equation}
We then estimate the quadratic variation. For any noise operator \(B\in \mathcal{L}_2(H,H^{-s})\), we denote by 
$$
T_B:=\|B\|_{\mathcal{L}_2(H,H^{-s})}^2,
$$and define its radial variance in the $D$-direction via, whenever $r_s(t)\neq 0$ (equivalently when $D\neq 0$)
$$
a_B(t):=\frac{\|B^*D(t)\|_{H^{}}^2}{r_s(t)^2}
.$$One can easily check that $a_B(t)\leq T_B$ for any $B$.
We then switch to the function 
\begin{equation}
f_\delta(r)=(r^2+\delta^2)^{p/2},
\qquad 0<p<1,
\label{4.17}
\end{equation}
so that $f_\delta(r)=\Phi_\delta(r^2)$. Then a simple calculation gives, for each $B$, 
\begin{equation}
\Phi_\delta'(R_s(t))T_B+2\Phi_\delta''(R_s(t))\|B^*D(t)\|_{H}^2=\frac{1}{2}f_\delta''(r_s(t))a_B(t)+\frac{1}{2}\frac{f_\delta'(r_s(t))}{r_s(t)}(T_B-a_B(t)).\label{4.13}
\end{equation}
Although this formula is written for \(D\neq 0\), one can check that it is still valid at $D=0$ and has the value $\frac{p}{2}\delta^{p-2}\|B\|^2_{\mathcal{L}_2(H,H^{-s})}$

\textbf{Step 5: Evaluating the noise term and concavity estimates.}
We next evaluate the noise terms.  Let $\iota:H\to H^{-s}$ be the natural injective map.

\textbf{The two rank-one projections.}
Denote by $e=e(t)=\frac{D(t)}{\|D(t)\|_2}$ so that $P_{D(t)}=e(t)\otimes e(t)$. Then $B_0(t)=m[g](1-c(t))\iota P_{D(t)}$ and $B_\perp(t)=-m[g]\sigma(t)\iota P_{D(t)}$, where $c(t)=\cos\theta_\delta(r_s(t))$ and $\sigma(t)=\sin\theta_\delta(r_s(t))$. We have that $\operatorname{Ran}B_0(t),\operatorname{Ran}B_\perp(t)\subset\operatorname{Span}\{D(t)\}$ so that this noise only acts on the direction of the difference process $U-V^n$, and we call this noise the radial noise.

For any $a\in\mathbb{R}$ and $B=a\iota P_{D(t)}$, we have that $Q^{1/2}P_{D(t)}=(Q^{1/2}e(t))\otimes e(t)$, so the rank-one Hilbert-Schmidt computations give that 
$$\|B\|^2_{\mathcal{L}_2(H,H^{-s})}=a^2
\|Q^{1/2}P_{D(t)}\|_{HS}^2=a^2\|e(t)\|_2^2\|Q^{1/2}e(t)\|_2^2=a^2\langle Qe(t),e(t)\rangle_2=\frac{a^2r_s(t)^2}{Z(t)}.
$$
The adjoint of $B=a\iota P_D:H\to H^{-s}$ is given as 
$
B^*z=aP_DQ,
$ so that 
$$
B^*D(t)=aP_{D(t)}QD(t)=a\langle QD(t),e(t)\rangle_2e(t)=a\frac{\langle QD(t),D(t)\rangle_2}{\sqrt{Z}(t)}e(t)=a\frac{r_s(t)^2}{\sqrt{Z}(t)}e(t).
$$Thus 
$$
\|B^*D(t)\|_H^2=a^2\frac{r_s(t)^4}{Z(t)}.
$$

Thus $T_{B_0(t)}=a_{B_0(t)}(t)$ and $T_{B_\perp(t)}=a_{B_\perp(t)}(t)$, and only the term involving $f_\delta''(r_s(t))$ remains in \eqref{4.13}, with coefficient
\begin{equation}
q_0(t)
=2m[g]^2(1-\cos\theta_\delta(r_s(t)))\frac{r_s(t)^2}{Z(t)}.
\label{4.16}
\end{equation}
Here and below \eqref{4.16} is invoked only for \(D\ne0\). If
\(r_s(t)=0\) (equivalently $D(t)=0\in H$), set \(a_B(t)=0\). Then the right-hand side of
\eqref{4.13} is
$
 \frac p2\delta^{p-2}T_B,
$
which agrees with the second-order Itô contribution at \(D(t)=0\).
Thus \eqref{4.13} extends to \(D(t)=0\) under these conventions.

\textbf{Derivative computation.} The derivatives of the function $f_\delta$ defined in \eqref{4.17} are

\begin{equation}
\frac{f_\delta'(r)}r
=p(r^2+\delta^2)^{p/2-1},
\label{4.18}
\end{equation}

and

\begin{equation}
f_\delta''(r)
=p(r^2+\delta^2)^{p/2-2}
\big(\delta^2+(p-1)r^2\big).
\label{4.19}
\end{equation}

Moreover, since $p\in(0,2)$,

\begin{equation}
f_\delta''(r)-\frac{f_\delta'(r)}r\le0.
\label{4.20}
\end{equation}

More explicitly, the left side of \eqref{4.20} is
\(p(p-2)r^2(r^2+\delta^2)^{p/2-2}\).  
Also, we have

\begin{equation}
f_\delta''(r)<0
\quad\Longleftrightarrow\quad
r>\frac{\delta}{\sqrt{1-p}}.
\label{4.21}
\end{equation}

\textbf{The residual part.}
The leading coefficient for $a_{B_F(t)}(t)$ (the
residual noise) in \eqref{4.13} is negative by \eqref{4.20} and can be discarded in an upper
bound. So the noise coefficient part $B_F(t)$ contribution is at most
$$
\frac{1}{2}\frac{f_\delta'(r_s(t))}{r_s(t)}T_{B_F(t)}
.$$
The term $T_{B_F(t)}$ is upper bounded by Lemma \ref{lemmas4.111} and we use the interpolation formula \eqref{4.3}.

The additive term is zero for \(r\le K_0\delta\).  Wherever the switch
is nonzero, \eqref{3.4} and \(K_0>2/\sqrt{1-p}\) imply \eqref{4.21}, so its entire
contribution is \(\frac12f_\delta''q_0\le0\).

\textbf{Conclusion.} Combining
\eqref{4.2}--\eqref{4.3}, \eqref{4.11}, and \eqref{4.13}--\eqref{4.21}, the predictable drift
\(\mathfrak b(t)\) in the Itô decomposition of \(f_\delta(r_s(t))\)
satisfies

\begin{equation}
\begin{aligned}
\mathfrak b(t)
\le{}&
-\frac{f_\delta'(r_s(t))}{r_s(t)}\mathcal A_s(t)
+\frac{f_\delta'(r_s(t))}{r_s(t)} r_s(t)^2\\
&+C\frac{f_\delta'(r_s(t))}{r_s(t)}
\big(r_s(t)^{2\beta(1-s)}\mathcal A_s(t)^{\beta s}
+\varepsilon_n^2\big)
+\frac12f_\delta''(r_s(t))q_0(t).
\end{aligned}
\label{4.22}
\end{equation}

\end{proof}

\section{Estimating the Itô drift}

The main objective of this section is to estimate the drift $\mathfrak{b}(t)$ in \eqref{4.22} in two different regions: very small $r_s(t)$ and slightly larger $r_s(t)$. We will also fix the correct parameter regime for $s,p$ for a valid estimate. The ultimate estimate we derive is \eqref{6.14}.

Choose, throughout the rest of the paper, a value of $s$ satisfying

\begin{equation}
\frac12<s<2-\frac1\beta.
\label{5.1}
\end{equation}
The interval is nonempty exactly when \(\beta>2/3\). Also, for
\(\beta<1\), one also has
\(2-1/\beta<\beta\), so the range \eqref{5.1} implies \(s<\beta\).

We denote by

\begin{equation}
q=\frac{\beta(1-s)}{1-\beta s}.
\label{5.2}
\end{equation}

Condition \eqref{5.1} is equivalent to \(q>1/2\).  Then choose a value of $p$ satisfying

\begin{equation}
2-2q<p<1.
\label{5.3}
\end{equation}

These choices are possible because

\begin{equation}
q>\frac12
\quad\Longleftrightarrow\quad
s<2-\frac1\beta,
\qquad
q<1\quad\Longleftrightarrow\quad\beta<1.
\label{5.3a}
\end{equation}

Here $s$ never takes the value of the two endpoints in \eqref{5.1}, and nor does $q$ and $p$ take endpoint values in their intervals.

\subsection{The inner region estimate}
We have the following estimate on the drift coefficient $\mathfrak{b}(t)$ in the regime where $r_s(t)$ is small. This lemma crucially gives the $\beta>\frac{2}{3}$ threshold, but does not strictly use non-degeneracy $m[g]>0$.
\begin{lemma}[inner-region estimate]\label{lemma5.1}
For the specified $q,s,p$, 
there is a constant \(C\) depending only on $g$ and thus independent of \(n,\delta\), such that, with $\mathfrak{b}(t)$ the drift defined in \eqref{4.22},

\begin{equation}
\mathfrak b(t)
\le C\left(\delta^p+\delta^\gamma
+\varepsilon_n^2\delta^{p-2}\right)
\quad\text{whenever }r_s(t)\le2K_0\delta,
\label{5.4}
\end{equation}

where

\begin{equation}
\gamma=p-2+\frac{2\beta(1-s)}{1-\beta s}
=p-2+2q>0.
\label{5.5}
\end{equation}

\end{lemma}
\begin{proof}

We denote by 
\begin{equation}
w_\delta(r)=\frac{f_\delta'(r)}r
=p(r^2+\delta^2)^{p/2-1},
\label{5.6}
\end{equation}
with the value at \(r=0\) defined by continuity.  If \(R_0=2K_0\) and
\(r\le R_0\delta\), then, because \(p-2<0\),

\begin{equation}
c_w\delta^{p-2}\le w_\delta(r)\le p\delta^{p-2},
\qquad
c_w=p(1+R_0^2)^{p/2-1}>0.
\label{5.7}
\end{equation}

Denote by $
\hat{a}=2\beta(1-s)$ and $\hat{b}=\beta s\in(0,1).$
Then the Hölder trace term in \eqref{4.22} is bounded by

\begin{equation}
\frac{f_\delta'(r_s(t))}{r_s(t)}
\big(r_s(t)^{2\beta(1-s)}\mathcal A_s(t)^{\beta s}
\big)\leq C\delta^{p-2+\hat{a}}{\mathcal {A}_s(t)}^{\hat{b}}.
\label{5.9}
\end{equation}

To keep track of every power, temporarily denote by \(\hat{X}=\delta^{p-2}\mathcal A_s(t)\).  Then

\begin{equation}
\delta^{p-2+\hat{a}}\mathcal A_s(t)^{\hat{b}}
=\delta^{\hat{a}+(p-2)(1-\hat{b})}\hat{X}^{\hat{b}}.
\label{5.10}
\end{equation}

For \(0<\hat{b}<1\), Young's inequality in the form

\begin{equation}
zX^{\hat{b}}\le\eta X+C_{\hat{b},\eta}z^{1/(1-\hat{b})}
\label{5.11}
\end{equation}
and the choice \(\eta=c_w/2\) give, for some $C'>0$ independent of $n$ and $\delta$,

\begin{equation}
 C\delta^{p-2+\hat{a}}\mathcal A_s(t)^{\hat{b}}
\le \frac{c_w}{2}\delta^{p-2}\mathcal A_s(t)
+C'\delta^{p-2+\hat{a}/(1-\hat{b})}.
\label{5.12}
\end{equation}
The last exponent is exactly \(\gamma\) in \eqref{5.5}.  The first term in
\eqref{5.12} is absorbed by half of the negative energy term
\(-w_\delta(r_s(t))\mathcal A_s(t)\) (which is the first term on the right hand side of \eqref{4.22}).  The remaining deterministic terms satisfy

\begin{equation}
w_\delta(r_s(t))r_s(t)^2\le C\delta^p,
\qquad
w_\delta(r_s(t))\varepsilon_n^2
\le C\varepsilon_n^2\delta^{p-2}.
\label{5.13}
\end{equation}

Finally we upper bound $f_\delta''(r_s(t))q_0(t)$. In the synchronous core $0<r_s(t)<K_0\delta$, we have $\theta_\delta(r_s(t))=0$, so that $q_0(t)=0$. When $r_s(t)=0$ we also have $q_0(t)=0$.  In the
transition annulus \(K_0\delta<r_s(t)<2K_0\delta\), \eqref{4.21} gives $f_\delta''(r_s(t))<0$ since $K_0>2/\sqrt{1-p}$, and thus
\(f_\delta''(r_s(t))q_0(t)\le0\).  Discard this negative term and the unused
half of the energy dissipation.  Equations \eqref{5.12}--\eqref{5.13} prove \eqref{5.4}.
\end{proof}
\subsection{The reflected outer region estimate}
The following lemma upper bounds the drift $\mathfrak{b}(t)$ in the active region $r_s(t)\geq 2K_0\delta$:
\begin{lemma}[outer-region absorption]\label{lemma6.1}
There is a deterministic \(r_*\in(0,1]\), and a constant $C_2>0$, both depending only on
\(m[g],c_g,[g]_{C^\beta},\beta,s,p\), such that the drift $\mathfrak{b}(t)$ in \eqref{4.22} satisfies 

\begin{equation}
\mathfrak b(t)
\le C_2 f_\delta(r_s(t))+C_2\varepsilon_n^2\delta^{p-2},\quad\text{whenever } 2K_0\delta\le r_s(t)\le r_*.
\label{6.1}
\end{equation}

\end{lemma}
This lemma does not reinforce $\beta>\frac{2}{3}$ but it uses $m[g]>0$ at the second estimate of \eqref{6.5}.
\begin{proof}

For \(r_s(t)\ge2K_0\delta\), the coupling is fully reflected in the sense that $K_\delta=I-2P_D$ and $S_\delta=0$.  Since
\(\delta/r_s(t)\le(2K_0)^{-1}\), \eqref{4.18}-\eqref{4.19} imply that
for constants $c_w,C_w,c_p$ depending only on $p,K_0$,
\begin{equation}
c_wr_s(t)^{p-2}\le w_\delta(r_s(t))\le C_wr_s(t)^{p-2},
\qquad
-f_\delta''(r_s(t))\ge c_pr_s(t)^{p-2},
\qquad
q_0=4m[g]^2\frac{r_s(t)^2}{Z(t)},
\label{6.2}
\end{equation}
where the strict positivity of \(c_p\) follows from $
(1-p)-\frac{\delta^2}{r_s(t)^2}
\ge (1-p)-\frac1{4K_0^2}>0.
$

Set

\begin{equation}
y_s(t)=\frac{\mathcal A_s(t)}{r_s(t)^2}\ge1,
\label{6.4}
\end{equation}
where the inequality \(y_s(t)\ge1\) follows directly from the Fourier weights. Then the
interpolation \eqref{4.3} reads \(Z(t)\le r_s(t)^2y_s(t)^s\).  Then the first and the fourth terms on the RHS of \eqref{4.22} satisfy
\begin{equation}
-w_\delta(r_s(t))\mathcal A_s(t)\le-cr_s(t)^py_s(t),
\qquad
\frac12f_\delta''(r_s(t))q_0(t)
\le-c\frac{r_s(t)^p}{Z(t)}
\le-cr_s(t)^{p-2}y_s(t)^{-s},
\label{6.5}
\end{equation}
where $c>0$ changes from line to line and does not depend on $\delta$. The second term on the RHS of \eqref{4.22} gives $cr_s(t)^p$.
The positive Hölder term (part of the third term on the RHS of \eqref{4.22}) is at most

\begin{equation}
\frac{f_\delta'(r_s(t))}{r_s(t)}
\big(r_s(t)^{2\beta(1-s)}\mathcal A_s(t)^{\beta s}
\big)
\le C r_s(t)^{p+2\beta-2}y_s(t)^{\beta s}.
\label{6.6}
\end{equation}

Thus \eqref{4.22} becomes

\begin{equation}
\begin{aligned}
\mathfrak b(t)
\le{}&
-c\big(r_s(t)^py_s(t)+r_s(t)^{p-2}y_s(t)^{-s}\big)
+Cr_s(t)^{p+2\beta-2}y_s(t)^{\beta s}\\
&+Cr_s(t)^p+Cp\varepsilon_n^2\delta^{p-2},
\end{aligned}
\label{6.7}
\end{equation}
where the last term comes from $p<2$ so that $w_\delta(r)\leq p\delta^{p-2}$.

We next show that the third term is negligible compared to the
two negative terms in the bracket. The comparison is stated in equation \eqref{6.9}, which is the ratio of the third term to the first two terms. 
\begin{equation}
\frac{r_s(t)^{2\beta-2}y_s(t)^{\beta s}}
{y_s(t)+r_s(t)^{-2}y_s(t)^{-s}}
\le  r_s(t)^{\frac{2(\beta-s)}{s+1}}.
\label{6.9}
\end{equation}
To prove this estimate, for \(0<r\le1\), set
$
y_0(r)=r^{-2/(s+1)}\ge1.$
Then for any \(1\le y_s(t)\le y_0(r_s(t))\), retain the term
\(r_s(t)^{-2}y_s(t)^{-s}\) in the denominator; and for \(y_s(t)\ge y_0(r_s(t))\), we instead retain the term \(y_s(t)\) in the denominator.
Then for $y_s(t)$ in both intervals, the maximum over $y$ is attained at
\(y_0(r_s(t))\).
 This verifies \eqref{6.9}.

Since \(s<\beta\), the exponent in \eqref{6.9} is strictly positive.  We then choose
\(r_*\le1\), independently of \(n,\delta\), such that \(Cr_*^{2(\beta-s)/(s+1)}\) is at most half the smaller
negative coefficient $c$ in the first term of RHS of \eqref{6.7}.  The third RHS term is then absorbed
whenever
\begin{equation}
2K_0\delta\le r_s(t)\le r_*.
\label{6.11}
\end{equation}
Finally, we also use \(r^p\le f_\delta(r)\).  Combining these estimates proves \eqref{6.1}.
\end{proof}

\subsection{Concluding drift estimate: combining inner and outer region}
Combining Lemma \ref{lemma5.1} and Lemma \ref{lemma6.1} gives the following stopped global
inequality:  

\begin{corollary}
Define the following stopping time
\begin{equation}
\tau_*^n=\inf\{t:r_s(U(t)-V^n(t))\ge r_*\},
\label{6.13}
\end{equation}
using the convention \(\inf\varnothing=\infty\).  Then for any \(t<\tau_*^n\), we have

\begin{equation}
\mathfrak b(t)
\le
C_3f_\delta(r_s(t))
+C_3\left(
\delta^p+\delta^\gamma
+\varepsilon_n^2\delta^{p-2}
\right),
\label{6.14}
\end{equation}
where $C_3>0$ depends on $m[g],c_g,[g]_{C^\beta},s,p,T$.

\end{corollary}
If \(2K_0\delta>r_*\), then every point before $\tau_n^*$ lies in the inner
region, so Lemma \ref{lemma5.1} alone proves \eqref{6.14}. Otherwise, we further use \ref{lemma6.1} for the outer region.

\section{Proof of weak uniqueness: the scalar case}

In this section, we finish the uniqueness part of Theorem \ref{theorem1.11}. This follows almost directly from \eqref{6.14}.

To make the small scale strictly positive, we take instead (although we do not really use the lower bound $n^{-1}$)
\begin{equation}
\delta_n=\max\{\varepsilon_n,n^{-1}\},
\label{7.1}
\end{equation}
and take the constant \(\delta=\delta_n\) in \eqref{3.7}.  Since \(D_0=0\) as the initial law is the same,
\begin{equation}
f_{\delta_n}(r_s(D_0))=\delta_n^p.
\label{7.2}
\end{equation}

\subsection{Uniqueness in law at each fixed time}
\begin{lemma}[Small error up to stopping via applying variational Itô formula]\label{lemmas7.1}

Let $\mathcal{E}_n$ denote the following small constant 

\begin{equation}
\mathcal{E}_n=\delta_n^p+\delta_n^\gamma
+\varepsilon_n^2\delta_n^{p-2}.
\label{7.3}
\end{equation}
Then, for every \(T<\infty\), we can find $C_T>0$ depending only on $T,p,s,c_g,[g]_{C^\beta}$ such that
\begin{equation}
\sup_{t\le T}\mathbb E f_{\delta_n}
\big(r_s(D(t\wedge\tau_*^n))\big)
\le C_T\mathcal{E}_n.
\label{7.4}
\end{equation}

\end{lemma}

The scale of $\mathcal{E}_n$ sends every error to zero. Indeed, since
\(\varepsilon_n\le\delta_n\),

\begin{equation}
\varepsilon_n^2\delta_n^{p-2}
=\delta_n^p\left(\frac{\varepsilon_n}{\delta_n}\right)^2
\le\delta_n^p.
\label{7.8}
\end{equation}
Thus \(\mathcal{E}_n\to0\), because \(\delta_n\to0\) and \(p,\gamma>0\).

\begin{proof}An estimate on the drift after applying Itô's formula is proven in \eqref{6.14}, and here we only rigorously verify the required integrability conditions for applying Itô, and then apply Gronwall at the end.

\textbf{Step 1: Computing the Fréchet derivative.}
For each fixed \(n\), the following function $F_n(D):D\in H^{-s}\to\mathbb{R}$,

\begin{equation}
F_n(D)=\big(\|D\|_{H^{-s}}^2+\delta_n^2\big)^{p/2}
\label{7.5}
\end{equation}
is globally \(C^2\) in the sense of Fréchet derivatives.  Its first Fréchet derivative satisfies 
$$
DF_n(D)[h]=(p/2)(\|D\|_{H^{-s}}^2+\delta_n^2)^{p/2-1}\cdot 2\langle D,h\rangle_{H^{-s},}
$$so that 
$$
\|DF_n(D)\|\leq p\|D\|_{H^{-s}}(\|D\|_{H^{-s}}^2+\delta_n^2)^{p/2-1}\leq \sup_{r\geq 0} pr(r^2+\delta_n^2)^{p/2-1}\leq C_p\delta_n^{p-1}, 
$$
Its second derivative is
$$\begin{aligned}
D^2F_n(D)[h,k]&=p(p-2)(\|D\|_{H^{-s}}^2+\delta_n^2)^{p/2-2}\langle D,k\rangle_{H^{-s}}\langle D,h\rangle_{H^{-s}}\\&+p(\|D\|_{H^{-s}}^2+\delta_n^2)^{p/2-1}\langle h,k\rangle_{H^{-s}}, 
\end{aligned}$$
so the norm of $D^2F_n(D)$ is bounded from above by \(C_p\delta_n^{p-2}\).

\textbf{Step 2: Verifying integrability and applying Itô.} As we have three independent noises, we denote the noise space by $\mathbb{U}=H\oplus H\oplus H$ and denote by $\mathfrak{B}_t^n:\mathbb{U}\to H^{-s}$ the (Hilbert-Schmidt) coefficient of the diffusion operator in the expression for the difference process $D(t)$ in \eqref{4.1a}. The cylindrical Wiener process is denoted $d\mathcal{W}$. Then applying variational Itô to $\|D\|_{H^{-s}}^2$ and  applying 1-d Itô to $\Phi_{\delta_n}(\|D\|^2_{H^{-s}})$ produces the following martingale 
$$
N_t^n=\int_0^{t\wedge\tau_*^n}DF_n(D(\ell))[\mathfrak{B}_\ell^nd\mathcal{W}_\ell].
$$
We claim that for some constant $C_{T,n}>0$ possibly depending on the law of $U$ and the approximation, $$\sup_{0\leq\ell\leq T}\mathbb{E}\|\mathfrak{B}_\ell^n\|_{\mathcal{L}_2(\mathbb{U},H^{-s})}^2\leq C_{T,n}<\infty.
$$
Indeed, this follows from the Hilbert-Schmidt norm estimate in \eqref{4.2} combined with the $L^2$-energy estimate of $D$ in \eqref{stronger4.8} (we do not yet need $n$- and $U$- independent estimates.)

Thus the martingale $N_t^n$ satisfies
\begin{equation}
\mathbb E\langle N^n\rangle_{T\wedge\tau_*^n}
\le\mathbb{E}\int_0^{T\wedge\tau_*^n}\|\mathfrak{B}_\ell^n\|_{\mathcal{L}_2(\mathbb{U},H^{-s})}^2\|\nabla F_n(D_\ell)\|_{H^{-s}}^2d\ell\leq 
C_{T,n}\delta_n^{2p-2}<\infty,
\label{7.6}
\end{equation}so the martingale $N_t^n$ is square integrable and $\mathbb{E}[N_t^n]=0$.
At this stage we only need this bound to be finite despite its $n$-dependence, as we only use it to prove that the stopped martingale has mean zero for each fixed \(n\).
For the drift term, the second order derivative term is similarly
$$
\mathbb{E}\int_0^{T\wedge \tau_*^n}|\operatorname{Tr}[(\mathfrak{B}_t^n)^*D^2F_n(D_t)\mathfrak{B}_t^n]|dt\leq C_p\delta_n^{p-2}\mathbb{E}\int_0^T\|\mathfrak{B}_t^n\|_{\mathcal{L}_2(\mathbb{U},H^{-s})}^2dt<\infty,
$$
which is also finite for each fixed $n$. Finally, the integrability of 
$$
\int_0^{T\wedge\tau_*^n}|DF_n(D_t)[\Delta D_t]|dt<\infty
$$follows from $|\langle\Delta D,D\rangle_{V^*,V}|\leq \mathcal{A}_s+r_s^2$ and then using the energy estimate in Lemma \ref{lemma4.237}.

\textbf{Step 3: the Gronwall argument.}
Consider the function
$$u_n(t):=\mathbb EF_n(D(t\wedge\tau_*^n))=\mathbb{E}f_{\delta_n}(r_s(t\wedge \tau_*^n)),$$we apply the stopped Itô identity
and the drift estimate \eqref{6.14} which shows that $$\mathbb{E}[\mathbf1_{t<\tau_*^n}\mathfrak{b}(t)]\leq C_3u_n(t)+C_3\mathcal{E}_n$$ as the martingale has zero expectation, to get that
\begin{equation}
u_n(t)\le u_n(0)+ C_3\mathcal{E}_nt+C_3\int_0^tu_n(\ell)d\ell.
\label{7.7}
\end{equation}
Here
\(\mathbf1_{\{\ell<\tau_*^n\}}F_n(D(\ell))
\le F_n(D({\ell\wedge\tau_*^n}))\).
Since $u_n(0)\leq\mathcal{E}_n$, then Gronwall finally proves \eqref{7.4}.
\end{proof}

From this we can prove that the stopping time is larger than $T$ with very high probability:

\begin{lemma}[No early stops]\label{lemmas7.2}
For every \(T<\infty\), we can find $C_T>0$ depending on $T,c_g,[g]_{C^\beta},\beta,s,p$ such that 

\begin{equation}
\mathbb P(\tau_*^n\le T)
\le C_Tr_*^{-p}\mathcal{E}_n\rightarrow0,\quad n\to\infty.
\label{7.9}
\end{equation}

\end{lemma}
\begin{proof}

The paths of \(D(t)\), solving the difference equation, are continuous in \(H^{-s}\), so that
\(r_s(D(\tau_*^n))=r_*\) on \(\{\tau_*^n\le T\}\).  Then we evaluate \eqref{7.4}
at \(t=T\).  Since \(F_n(D)\ge r_s(D)^p\) for any $D\in H^{-s}$,
\begin{equation}
r_*^p\mathbb P(\tau_*^n\le T)
\le\mathbb EF_n(D({T\wedge\tau_*^n}))
\le C_T\mathcal{E}_n,
\label{7.10}
\end{equation}
which proves \eqref{7.9} since $
\mathcal{E}_n\to 0$.
\end{proof}

\begin{corollary}\label{corollary7.3}
Lemma \ref{lemmas7.1} and \ref{lemmas7.2} imply that, for every fixed \(t>0\),

\begin{equation}
r_s(U(t)-V^n(t))\longrightarrow0
\quad\hbox{in probability}.
\label{7.11}
\end{equation}
\end{corollary}
\begin{proof}
Indeed, for every \(\eta>0\),

\begin{equation}
\begin{aligned}
\mathbb P(r_s(D(t))>\eta)
\le{}&\mathbb P(\tau_*^n\le t)+\eta^{-p}\mathbb E\!\left[
r_s(D({t\wedge\tau_*^n}))^p;\tau_*^n>t\right],
\end{aligned}
\label{7.12}
\end{equation}
the first term tends to zero by Lemma \ref{lemmas7.2} and the second term tends to zero by Lemma \ref{lemmas7.1}.  
\end{proof}

For later convenience, we restate the convergence in probability in Corollary \ref{corollary7.3} into the following bounded expectation convergence.
We introduce the following bounded metric
$$d_s(h,k):=1\wedge r_s(h-k)=1\wedge \|h-k\|_{H^{-s}},\quad h,k\in H^{-s},$$ and the elementary inequality
\(1\wedge r\le r^p\), which is valid for all \(0<p<1\), then gives

\begin{equation}
\begin{aligned}
\mathbb E d_s(U(t),V^n(t))
&\le \mathbb P(\tau_*^n\le t)\\
&\quad+\mathbb E\!\left[
r_s(D({t\wedge\tau_*^n}))^p;\tau_*^n>t\right]
\rightarrow0.
\end{aligned}
\label{7.13}
\end{equation}
(Indeed, on the event $\tau_*^n\leq t$ the distance is upper bounded by $d_s\leq 1$, and on the complementary event we use $1\wedge r\leq r^p$.)
Importantly, the rate of convergence to 0 of this quantity depends only on the coefficient $g$ and $s,p$ and $\varepsilon_n$, and does not depend on
the chosen weak law of \(U\).  For every candidate
solution law \(P\) to the SPDE \eqref{1.1}, let
\(\mu_t^P=\operatorname{Law}_P(U(t))\), and let \(\mu_t^n\) be the
time-\(t\) marginal of \(P_n\) \eqref{unequation}, which coincides with the law of the solution to \eqref{3.7} by Lemma \ref{preservations}.  The full joint law constructed in
Sections ~\ref{sections2} and \ref{sections3} is an admissible coupling, so \eqref{7.13} implies

\begin{equation}
W_{d_s}(\mu_t^P,\mu_t^n)\longrightarrow0,
\qquad
d_s(x,y)=1\wedge\|x-y\|_{H^{-s}},
\label{7.14}
\end{equation}
where $W_{d_s}$ is the Wasserstein distance defined on the space of probability measures on $H^{-s}$ with respect to the distance function $d_s$.
Moreover, if \(P,Q\) are two candidate weak-solution laws to the SPDE \eqref{1.1}, then the above computation gives

\begin{equation}
W_{d_s}(\mu_t^P,\mu_t^Q)
\le W_{d_s}(\mu_t^P,\mu_t^n)
+W_{d_s}(\mu_t^n,\mu_t^Q)\longrightarrow0.
\label{7.15}
\end{equation}
Thus the law of every candidate weak solution $\mu^P$ to the original SPDE \eqref{1.1} has the same time-\(t\) marginal for any $t>0$.  This justifies
that the law is unique at every endpoint. In the next section we prove that this uniqueness of law at each time $t$ implies
the uniqueness of the entire path law, hence completing the whole weak uniqueness statement.

\subsection{Uniqueness from one-time marginals to path law}

Finally, we upgrade uniqueness at each time to uniqueness on the continuous path space.

Recall that $H=L^2(\mathbb{T})$.
Let
\(\Omega_T=C([0,T];H)\), and let \(X\) be the coordinate process.  Let $\mathcal{C}\subset C^\infty(\mathbb
T)$ be the span of the bases $\{e_k\}$, which involve the constant function and the real trigonometric polynomials. For any $x\in H$, we define the martingale problem $\operatorname{MP}(g,x)$. A probability law \(P\)
on \(\Omega_T\) solves \(\mathrm{MP}(g,x)\) if \(P(X_0=x)=1\) and,
for every \(\phi,\psi\in\mathcal C\), the process
\begin{equation}
M_t^\phi
:=\langle X_t,\phi\rangle-\langle x,\phi\rangle
-\int_0^t\langle X_r,\Delta\phi\rangle dr
\label{8.1}
\end{equation}
is a continuous martingale with quadratic covariation
\begin{equation}
[M^\phi,M^\psi]_t
=\int_0^t\int_{\mathbb T}
g(X_r(y))^2\phi(y)\psi(y)\,dy\,dr.
\label{8.2}
\end{equation}
See for example \cite[Section~2]{Walsh1986} for more background on the martingale problem formulation.

The equivalence between martingale problems and weak solutions to SHEs is well known:
\begin{lemma}[equivalence with the weak SPDE]\label{equivalence2}

The laws of weak mild solutions to the SHE \eqref{1.1} have a one-to-one correspondence with the solutions
of \(\mathrm{MP}(g,u_0)\).

\end{lemma}

\begin{proof}
    The infinite-dimensional correspondence
between stochastic evolution equations and their associated martingale
problems can be found in
\cite[Theorem 3.6 and Corollary 3.7]{Kunze2013}. The energy estimates in Lemma \ref{lemma4.237} verify all necessary conditions for the use of that theorem. 
\end{proof}

\begin{lemma}[weak existence]\label{weakexist}
For every \(x\in H\), \(\mathrm{MP}(g,x)\) is nonempty.
\end{lemma}
\begin{proof}
By the equivalence in Lemma \ref{equivalence2}, the existence of MP follows from the weak existence result of SHE \eqref{1.1} in \cite{gkatarek1994weak}, because the coefficient $g$ is continuous and of linear growth.
\end{proof}

Let \(\iota:H\to H^{-s}\) be the continuous inclusion.  For the smooth
equation with solution $V^n$ and coefficient $g_n$, let \(\Pi_t^n(u_0,\cdot)\) be its endpoint probability transition kernel on \(H\), which gives the law of $V^n$ at time $t$ with initial value $u_0$.
Since $g_n$ is globally Lipschitz, this gives a strong solution measurable in
\(u_0\in H\), so \(u_0\mapsto \iota_{\# } \Pi_t^n(u_0,\cdot)\) is a Borel map from $H$ into the
Polish space \(\mathcal P(H^{-s})\).

For each fixed \(u_0\), Lemma \ref{weakexist} guarantees at least one candidate weak
solution to $\operatorname{MP}(g,u_0)$.  Then estimate \eqref{7.14}, applied from this \(u_0\), shows that \footnote{ \eqref{8.12} is interpreted as the convergence of probability law from the left to the right in Wasserstein distance $W_{d_s}$.}

\begin{equation}
\iota_{\# } \Pi_t^n(u_0,\cdot)\longrightarrow\overline P_t(u_0,\cdot)
\quad\text{in }W_{d_s},
\label{8.12}
\end{equation}
where $\overline P_t(u_0,\cdot)$ is the probability transition kernel of one martingale solution, but as the convergence holds for every MP solution (while the left hand side is independent of the MP solution), this makes the limit independent of the chosen
candidate, so that $\overline P_t(u_0,\cdot)$ is the same martingale transition kernel for all elements in $\operatorname{MP}(g,u_0)$. The map \(u_0\mapsto\overline P_t(u_0,\cdot)\) is a Borel probability
kernel from $H$ to \(H^{-s}\). Since each solution of the martingale problem $\operatorname{MP}(g,u_0)$ is \(H\)-valued, and the injection image $\iota (H)\in H^{-s}$ is a Borel subset by Lusin–Souslin, and the pullback map $\iota^{-1}:\iota (H)\to H$ is also Borel, we must have that 
\(\overline P_t(u_0,\iota (H))=1\). Then pulling back by \(\iota^{-1}\) now defines a
Borel kernel \(P_t(u_0,\cdot)\) on \(H\).  In this time-homogeneous setting we then 
write \(P_{r,t}=P_{t-r}\).  This is the common endpoint law from every
starting pair \((r,u_0)\).

To control multi-dimensional time marginals, we use the following: 
\begin{lemma}[conditional MP]
\label{condiMP}
Let \(P\) solve \(\mathrm{MP}(g,u_0)\), let \(r<T\) be deterministic,
and let \(P_\omega^r\) be a regular conditional law of
\((X_{r+t})_{0\le t\le T-r}\) given the canonical
\(\mathcal F_r^X\).  Then, for \(P\)-almost every \(\omega\),

\begin{equation}
P_\omega^r\in\mathrm{MP}(g,X_r(\omega)).
\label{8.13}
\end{equation}

\end{lemma}
This is the standard stability of martingale problems under conditioning; see
\cite[Theorem 1.2.10]{stroock2007multidimensional}.

\begin{lemma}[From endpoint uniqueness to path uniqueness]\label{omittedfinallemma}
Under the assumptions of $g$ in Theorem \ref{theorem1.11},
every solution of \(\mathrm{MP}(g,u_0)\) has the same law on
\(C([0,T];H)\).

\end{lemma}
\begin{proof}

By Lemma \ref{condiMP} and \eqref{8.12}, for every bounded continuous function
\(\varphi:H\to\mathbb R\) and \(0\le r<t\le T\), and for every solution $P\in\operatorname{MP}(g,u_0),$
\begin{equation}
\mathbb E_P[\varphi(X_t)\mid\mathcal F_r^X]
=P_{r,t}\varphi(X_r),
\label{8.16}
\end{equation}
while the right hand side does not depend on the solution $P$.
Since the space of bounded continuous functions is convergence-determining and we can find a countable convergence-determining family, this implies that 
$$
\operatorname{Law}(X_t\mid\mathcal{F}_r^X)=P_{r,t}(X_r,\cdot).
$$
Thus every solution $P$ to $
\operatorname{MP}(g,u_0)$ is Markov
with the same transition kernels.  Iteration gives, for
\(0<t_1<\cdots<t_k\le T\),
\begin{equation}
\operatorname{Law}(X_{t_1},\ldots,X_{t_k})
=\delta_{u_0}P_{0,t_1}P_{t_1,t_2}\cdots P_{t_{k-1},t_k}.
\label{8.17}
\end{equation}
All finite-dimensional distributions are therefore common.  The Borel
sigma-field of \(C([0,T];H)\) is generated by evaluations at rational
times, so the full canonical path laws coincide for any $P$.\end{proof}

Lemma \ref{omittedfinallemma} completes the proof of Theorem \ref{theorem1.11}.

\section{Proof of the two extensions}
In this section we prove the two extensions of Theorem \ref{theorem1.11}.

\subsection{The vector valued case}We prove Theorem \ref{thm:vector-valued}. We first fix a constant \(m[G]\in(0,c_G/4)\)
and define
$$
    R(z):=
    \bigl(G(z)G(z)^{*}-m[G]^2I_d\bigr)^{1/2}.
$$
\begin{lemma}
\label{holderlemmas} Let $G$ satisfy the assumptions in Theorem  \ref{thm:vector-valued}. Then we can find a constant \(L_G<\infty\) depending only on $G$ such that,
$$ 
    \|R(z)-R(z')\|_{\mathrm{HS}}
       \leq L_G|z-z'|^\beta,
    \qquad \forall z,z'\in\mathbb R^d.
$$
\end{lemma}

\begin{proof}
[\proofname\ of Lemma \ref{holderlemmas}]    
Set \(P(z)=(G(z)G(z)^*)^{1/2}\). By the
Araki–Yamagami inequality \cite{araki1981inequality} for the matrix absolute-value map,
\begin{equation}
    \|P(z)-P(z')\|_{\mathrm{HS}}
       \leq \sqrt{2}\,
       \|G(z)-G(z')\|_{\mathrm{HS}}.\label{whatismoreover}
\end{equation}
Moreover, \(R(z)=f_G(P(z))\), where
\[
    f_G(r)=\sqrt{r^2-m[G]^2},\qquad r\geq c_G.
\]
Since $f_G$ is Lipschitz continuous on $[c_G,\infty)$ with Lipschitz constant $\frac{c_G}{\sqrt{c_G^2-m[G]^2}}$ and $P(z),P(z')$ are normal operators, the Hilbert--Schmidt functional-calculus
estimate for normal operators (see \cite[Corollary~2]{Kittaneh1985}) applies and we get 
\[
    \|R(z)-R(z')\|_{\mathrm{HS}}
       \leq
    \frac{c_G}{\sqrt{c_G^2-m[G]^2}}
    \|P(z)-P(z')\|_{\mathrm{HS}}.
\]Combined with \eqref{whatismoreover} and the assumption on $G$,
the conclusion follows.
\end{proof}

\begin{proof}[\proofname\ of Theorem \ref{thm:vector-valued}]
For the vector field case, we denote by, for any $\alpha\in \mathbb{R}$,
$$
    \mathcal H=L^2(\mathbb T;\mathbb R^d),
    \qquad
    \mathcal H^\alpha=H^\alpha(\mathbb T;\mathbb R^d),
$$
with all Sobolev norms understood componentwise. We indicate
only the changes from the scalar proof. Let
\[
P(z)=\bigl(G(z)G(z)^*\bigr)^{1/2}.
\]
Since \(G(z)G(z)^*\geq c_G^2I_d\), the matrix
\[
    O(z)=P(z)^{-1}G(z)
\]
is orthogonal. This orthogonal transform of the white noise does not change the law, so it reduces the original equation to one with diffusion
coefficient \(P(z)\) and white noise $d\widetilde W=O(\cdot)dW$. We henceforth assume that the SHE has diffusion coefficient $P(z)$ which is a self-adjoint matrix.

Fix an arbitrary weak solution \((U,W)\).
Let $W'$ be another independent white noise. Since by definition,
$P(z)^2=m[G]^2I_d+R(z)^2,
$
the noise may further be represented by two independent
\(d\)-dimensional space--time white noises \(W^0,W^1\) via the following decomposition
$$
dW^0=m[G]P(U)^{-1}d\widetilde W+R(U)P(U)^{-1}dW',
$$

$$
dW^1=R(U)P(U)^{-1}d\widetilde W-m[G]P(U)^{-1}dW'.
$$

We can then verify that they are independent and
\begin{equation}
  G(U)\,d W
       \ {=}   P(U)\,d\widetilde W
       \ {=}\
    m[G]\,dW^0+R(U)\,dW^1.
    \label{V.1}
\end{equation}
By Lemma \ref{holderlemmas}, $R(z)$ is again uniformly $\beta$-Hölder continuous.

Choose smooth globally Lipschitz approximating maps 
$R_n:\mathbb R^d\to\mathbb R^{d\times d}
$
such that
\[
    \varepsilon_n:=\sup_{z\in\mathbb R^d}
       \|R_n(z)-R(z)\|_{\mathrm{HS}}\longrightarrow0.
\]We choose \(R_n\) by convolution with a nonnegative mollifier, so that
\(R_n\) remains symmetric. Let
\[
P_n(z)=\bigl(m[G]^2I_d+R_n(z)^2\bigr)^{1/2}.
\]
The same Hilbert--Schmidt functional-calculus estimate shows that
\(P_n\) is globally Lipschitz and uniformly elliptic. 
The reflection coupling is then defined exactly as in the
scalar proof, with the reflection operators acting on the
Hilbert space \(\mathcal H\). Namely, for
\(D=U-V^n\), set
\[
    e(D)=\frac{D}{\|D\|_{\mathcal H}},
    \qquad
    P_D=e(D)\otimes e(D),
\]
and use the same operators \(K_\delta(D)\) and
\(S_\delta(D)\) as in \eqref{3.5}. The equation for \(V^n\) is
\begin{equation}
\begin{split}
    dV^n={}&\Delta V^n\,dt
       +m[G]K_\delta(D)\,dW^0
       +m[G]S_\delta(D)\,d\overline W^0  \\
       &\quad +M_{R_n(V^n)}\,dW^1,
\end{split}
\label{V.2}
\end{equation}
where \(M_F\) denotes pointwise multiplication by the
matrix-valued function \(F\). By the same
bracket calculation as in Lemma~\ref{preservations}, \(V^n\) has the
canonical law of the \(P_n\)-equation, independently of \(U\).

All deterministic Sobolev interpolation estimates remain
unchanged. The trace inequality is now written more generally as, in place of \eqref{4.4},
\begin{equation}
    \bigl\|A^{-s/2}M_F
       \bigr\|_{\mathcal L_2(\mathcal H)}^2
      =
    \kappa_s\int_{\mathbb T}
        \|F(x)\|_{\mathrm{HS}(\mathbb R^d)}^2\,dx,
    \qquad s>\frac12.
    \label{V.3}
\end{equation}
Consequently, the Hölder assumption on $G$ gives, in place of Lemma \ref{lemmas4.111},
\begin{equation}
\begin{split}
 &\bigl\|A^{-s/2}
       M_{R(U)-R_n(V^n)}
   \bigr\|_{\mathcal L_2(\mathcal H)}^2       \\
 &\qquad\leq
   C\int_{\mathbb T}|U(x)-V^n(x)|^{2\beta}\,dx
     +C\varepsilon_n^2 \leq
   C\|D\|_{\mathcal H}^{2\beta}
     +C\varepsilon_n^2.
\end{split}
\label{V.4}
\end{equation}
The interpolation inequality
\begin{equation}
    \|D\|_{\mathcal H}^2
       \leq
    \|D\|_{\mathcal H^{-s}}^{2(1-s)}
    \|D\|_{\mathcal H^{1-s}}^{2s}
    \label{V.5}
\end{equation}
is the same as the scalar version \eqref{4.3}.

Equations \eqref{V.3}--\eqref{V.5} are precisely the
inequalities used in the variational Itô formula estimate. All the later estimates are unchanged after replacing
\(L^2(\mathbb T)\) by \(\mathcal H\).  In particular, the
closing condition remains
$ \beta>\frac23.$ The rest of the arguments are identical to the scalar proof and yield equality in
law of any weak solution.
\end{proof}

\subsection{Extension to a globally Lipschitz drift}
\begin{proof}[\proofname\ of Corollary \ref{cor:lipschitz-drift}]
We use the same coupling, adding \(b(U(t))\) and
\(b(V^n(t))\) to the two coupled equations. The standard energy
estimates remain valid because \(b\) has linear growth. In the Itô computation involving $f_\delta(r_s(t))$, in the inner region (Lemma \ref{lemma5.1}) its only additional contribution satisfies (denote by $\rho_\delta=(r_s^2+\delta^2)^{1/2}$ and $w_\delta=p\rho_\delta^{p-2}$),
\[
\begin{aligned}
&\rho_\delta^{p-2}
 \bigl|
 \langle D,b(U)-b(V^n)\rangle_{H^{-s}}
 \bigr|                                                \\
&\qquad\leq
 L_b\rho_\delta^{p-2}\|D\|_{H^{-s}}\|D\|_2             \\
&\qquad\leq
 L_b\rho_\delta^{p-s}\|D\|_{H^{1-s}}^s    \qquad\text{(using interpolation and $\|D\|_{H^{-s}}\leq\rho_\delta$)}             \\
&\qquad\leq
 \varepsilon\rho_\delta^{p-2}\|D\|_{H^{1-s}}^2
 +C_{\varepsilon,L_b}\rho_\delta^p.
\end{aligned}
\]
The first term is absorbed by the dissipative Laplacian term and the
second by the existing Gronwall estimate. For the inner region we use the same estimate \eqref{5.11}. For the outer region (Lemma \ref{lemma6.1}), the contribution of $
 \langle D,b(U)-b(V^n)\rangle_{H^{-s}}
  $ is similarly absorbed in the negative energy term via
  $$
w_\delta|\langle D,b(U)-b(V^n)\rangle_{H^{-s}}|\lesssim r_s^py_s^{s/2}\lesssim \eta r_s^py_s+C_\eta r_s^p
  $$
  and does not change the later computations. The remaining preservation, stopping, endpoint uniqueness and martingale problem proofs remain unchanged.
\end{proof}

\section*{Funding}
The majority of this work was completed while the author was affiliated with IAS Princeton, during which the author was supported by a fellowship from IAS provided by the S.S. Chern Foundation for Mathematical Research Fund and the Fund for Mathematics. 

\section*{Declaration of generative AI usage}

The author used OpenAI's ChatGPT as an assistance tool for language
editing and exploratory checking of some calculations. All mathematical
arguments were independently verified by the author, who takes full
responsibility for the contents of the paper.

\printbibliography

\end{document}